\documentclass[a4paper, 12pt]{amsart}
\usepackage[T1]{fontenc}
\usepackage{xspace}
\usepackage{enumerate}
\usepackage[utf8]{inputenc} 
\usepackage[english]{babel}
\usepackage{lmodern}
\usepackage{microtype}
\usepackage{amssymb}
\usepackage{verbatim}
\usepackage{amsthm}
\usepackage{amsmath}
\usepackage{mathtools}
\usepackage{color}
\usepackage{tikz}
\usepackage{threeparttable}
\usepackage{graphicx}
\usepackage{caption}
\usepackage{paralist}
\usepackage[pagebackref=true, pdfborder={0 0 0}, linktocpage]{hyperref}
\usepackage{backref}
\usepackage{aliascnt}
\usepackage{cleveref}
\usepackage{dsfont}
\usepackage{mathrsfs}
\usepackage{bm}

\newcommand{\R}{\mathbb{R}}

\renewcommand{\P}{\mathbb{P}}
\newcommand{\E}{\mathbb{E}}

\newcommand{\cC}{\mathcal{C}}

\newcommand{\cK}{\mathcal{K}}
\newcommand{\cL}{\mathcal{L}}

\newcommand{\cT}{\mathcal{T}}

\newtheorem{theorem}{Theorem}[section]
\newtheorem{lemma}[theorem]{Lemma}

\newtheorem{assum}[theorem]{Assumption}
\theoremstyle{remark}
\newtheorem{remark}[theorem]{Remark}

\numberwithin{equation}{section}

\begin{document}

\title[Asymptotics of ESRF in discrete and continuous time]{Mean field limit of Ensemble Square Root Filters - discrete and continuous time}
\author{Theresa Lange}
\address{Institut f\"ur Mathematik, Technische Universit\"at Berlin, Stra{\ss}e des 17. Juni 136, D-10623 Berlin}
\email{tlange@math.tu-berlin.de}

\author{Wilhelm Stannat}
\address{Institut f\"ur Mathematik, Technische Universit\"at Berlin, Stra{\ss}e des 17. Juni 136, D-10623 Berlin and Bernstein Center for Computational Neuroscience, Philippstr. 13, D-10115 Berlin}
\email{stannat@math.tu-berlin.de}

\date{Berlin, January 04, 2021}

\begin{abstract}
Consider the class of Ensemble Square Root filtering algorithms for the numerical approximation of the posterior distribution of nonlinear Markovian signals, partially observed with linear observations corrupted with independent measurement noise. We analyze the asymptotic behavior of these algorithms in the large 
ensemble limit both in discrete and continuous time. We identify limiting mean-field processes on the level of the ensemble members, prove corresponding propagation of chaos results and derive associated convergence rates in terms of the ensemble size. In continuous time we also identify the stochastic partial differential equation driving the distribution of the mean-field process and perform a comparison with the Kushner-Stratonovich equation. 
\end{abstract}


\maketitle 

\section{Introduction}\label{intro}
\noindent
Ensemble Square Root Filters (ESRF) belong to the class of ensemble-based Kalman-type filtering algorithms which specify evolution equations of an ensemble inspired by the Kalman Filter \cite{kalman1960}. Filtering algorithms in general aim at approximating the current state of a partially known system with the help of observations. Precisely, consider the following setting: 
\begin{align*} 
X_k &= B\left(X_{k-1}\right) + CW_k,\\
Y_k &= HX_k + \Gamma V_k
\end{align*}
where the $d$-dimensional dynamical system $X$ is observed by the $q$-dimensional process $Y$. Both $W$ and $V$ are independent Gaussian processes denoting model and measurement error with positive definite covariance matrices $Q:= CC^T$ and $R:= \Gamma \Gamma^T$, respectively, and $B:\mathbb{R}^{d} \rightarrow \mathbb{R}^{d}$ is assumed to be globally Lipschitz, as well as $H \in \mathds{R}^{q\times d}$. The problem is to identify (or at least approximate) the conditional distribution
\[ 
\pi_k(x) := \mathds{P}\left[X_k = x | \mathcal{Y}_{0:k}\right] 
\]
where $\mathcal{Y}_{0:k} := \sigma\left(Y_0, ..., Y_k\right)$ is the sigma algebra generated by all past and the current observation. \\
Recall that if $B$ is linear and $\pi_0$ is Gaussian, then also $\pi_k$ is Gaussian, reducing the computation of $\pi_k$ to the well-known Kalman Filter specifying recursive equations for mean and covariance matrix.\\
\noindent
In the nonlinear setting, various Monte-Carlo simulation approaches have been developed in order to find an estimate for the in this case non-Gaussian $\pi$. In contrast to the class of particle filters approximating the full distribution, ensemble-based Kalman-type filtering algorithms propose an ensemble-version of the above Kalman filtering equations thus specifying only an estimate of first and second moment of $\pi_k$.\\
Whereas the asymptotic behavior of classical sequential Monte-Carlo approximations has been studied thoroughly in the literature, the large ensemble limit of the above class of ensemble-based algorithms is yet largely unexplored apart from e.g. \cite{leGland2009}, \cite{mandel2011}, \cite{kwiatkowski2015}, or \cite{abishop2020b} and references therein. Intuitively, assuming a linear, Gaussian setting, one would expect that ensemble-based Kalman-type filtering algorithms admit a consistent mean field limit due to their close relationship to the Kalman Filter.\\
In this paper, our investigations will focus on the class of ESRF 
algorithms which are widely used in the geosciences (see e.g. \cite{anderson2001}, \cite{leeuwenburgh2005}, \cite{okane2008}, \cite{beyou2013}). Popular examples are the Ensemble Adjustment Kalman Filter (EAKF, see \cite{anderson2001}), the Ensemble Transform Kalman Filter (ETKF, see \cite{bishop2001}), and the unperturbed EnKF (Whitaker, Hamill (2002), see \cite{whitaker2002}), as surveyed in \cite{tippett2003}. With the exceptions of a few works such as \cite{kwiatkowski2015}, these algorithms lack an in-depth analysis in the mean field literature so far.\\
The class of ESRF divides into transformations on the signal space (adjustment filters) and on the ensemble space (transform filters). The latter cannot be interpreted in a mean field sense in general since the associated transformations of these filters will not be homogeneous over the whole ensemble. Nevertheless, we can identify conditions on transform filters to ensure a uniform treatment of the ensemble and hence allow for a mean field limit, as made precise in Assumption \ref{AssumT} in Section \ref{algoDiscr}.\\
Thus in the discrete-time setting, we prove results on the asymptotic behavior for this class of filtering algorithms and characterize their limiting mean field processes on the level of the ensemble members. This distinguishes our work from the existing literature, e.g. \cite{kwiatkowski2015} specifying in the setting of linear, perfect models a mean field limit solely for first and second moment. In particular in Theorem \ref{thm:MeanFieldNonlinear}, we provide a corresponding propagation of chaos result, which we complement with the derivation of convergence rates in Theorem \ref{error}. Finally we investigate consistency of the mean field limit process which holds true in the linear and Gaussian setting.\\
\noindent
We also consider the continuous-time filtering problem
\begin{align*}
{\rm d}X_t &= B\left(X_t\right){\rm d}t + Q^{\frac{1}{2}}{\rm d}W_t,\\
{\rm d}Y_t &= HX_t{\rm d}t + C^{\frac{1}{2}}{\rm d}V_t.
\end{align*}
Motivated by the continuous time limit analysis of discrete-time ESRF algorithms conducted in \cite{lange2019}, we investigate in Section \ref{secMeanCont} the asymptotic behavior of ensemble-based Kalman-Bucy-type filtering algorithms initially proposed in \cite{bergemann2012}, which we will introduce in more detail in Section \ref{algoCont}. In the simpler case of $H = {\rm Id}$, a mean field analysis has been conducted in \cite{deWiljes2018}. We provide corresponding mean field results in the case of general matrices $H$ as made precise in Theorem \ref{resultMeanCont}, investigate the linear case and discuss the relationship of the stochastic partial differential equation driving the distribution of the mean-field process and the Kushner-Stratonovich equation.

\subsection{Notation}
Throughout the paper, we will be using standard notation: for a vector $x \in \mathds{R}^{n}$ and a matrix $A \in \mathds{R}^{n \times m}$ let $x^T$ and $A^T$ denote the vector transpose and the matrix transpose, respectively. Further, let $\|x\|$ denote a vector norm on Euclidean space $\mathds{R}^n$ as well as $\|A\|$ denote the spectral norm and $\text{tr}(A)$ the trace of $A$. For a Lipschitz-continuous function $f$, we use $\|f\|_{\text{Lip}}$ for the Lipschitz constant. Also let $\mathds{R}^{d\times d}_{sym ,+}$ denote the set of symmetric positive semidefinite matrices in $\mathds{R}^{d \times d}$.\\
Additionally, we will abbreviate "independent and identically distributed" by "i.i.d.", as well as "almost surely with respect to the probability measure $\mathds{P}$" by "$\mathds{P}$-a.s.". 

\section{Algorithms}\label{algo}
\noindent

\subsection{Discrete time}\label{algoDiscr}
\noindent
Consider the discrete-time setting
\begin{align}
X_k &= B\left(X_{k-1}\right) + CW_k,\label{DiscrX}\\
Y_k &= HX_k + \Gamma V_k\label{DiscrY}
\end{align}
where $W_k$, $V_k$, $k \in \mathds{N}$, are independent Gaussian random variables and $C$ and $\Gamma$ are such that $Q:= CC^T$ and $R:= \Gamma \Gamma^T$ are positive definite. As indicated in the introduction, ESRF algorithms consists of two steps that are iterated in order to obtain an approximation of mean and covariance matrix of the posterior distribution $\pi_k$ with the help of the empirical mean and covariance 
\begin{align*}
\bar{x}_k^{a} &:= \frac{1}{M} \sum_{i=1}^M X_k^{(i),a},\\
P_k^a &:= \frac{1}{M-1} \sum_{i=1}^M \left(X_k^{(i),a} - \bar{x}_k^a\right)\left(X_k^{(i),a}  
- \bar{x}_k^a\right)^T 
\end{align*}
of some ensemble of $M$ particles $X_k^{(i),a}$, $1\leq i\leq M$.\\
\noindent
In the \textit{forecast} step, the ensemble members $\left(X^{(i),a}_{k-1}\right)_{1\leq i\leq M}$ from 
the previous estimation cycle are propagated according to (\ref{DiscrX}):  
\begin{equation} 
\label{ForwardStep} 
X_k^{(i),f} := B\left(X_{k-1}^{(i),a}\right) + CW_k^{(i)}
\end{equation} 
where $W_k^{(i)}$, $1\leq i\leq M$, are independent copies of $W_k$, and superscript $f$ denotes "forecast". Let  
\begin{align*}
\bar{x}_k^{f} &:= \frac{1}{M} \sum_{i=1}^M X_k^{(i),f},\\
P_k^f &:= \frac{1}{M-1} \sum_{i=1}^M \left(X_k^{(i),f} - \bar{x}_k^f\right) \left(X_k^{(i),f}  
- \bar{x}_k^f\right)^T 
\end{align*}
be the corresponding empirical mean and covariance matrix, as well as  
\[E_k^f :=   \left[X_k^{(i),f} - \bar{x}_k^f\right]_{i=1}^M \]
be the $d\times M$-matrix with the centered forecast ensemble members as column vectors.\\
In the \textit{update} step, the ESRF transforms the forecast ensemble into an \textit{updated}, 
or \textit{analyzed} ensemble $X_k^{(i),a}$, $1\leq i\leq M$, (indicated by superscript $a$) taking into account the new observation 
$Y_k$, according to the following ansatz:  
\begin{equation} 
\label{AnalysisStep} 
X_k^{(i),a} := \bar{x}_k^a + E_k^a e_i
\end{equation} 
where $\bar{x}_k^a$, the empirical mean of the analyzed ensemble members, is defined as 
\begin{equation} 
\label{EmpMeanAnalyzedParticles} 
\bar{x}_k^a := \bar{x}_k^f + \cK\left(P_k^{f}\right) (Y_k - H\bar{x}_k^f )
\end{equation}
with Kalman gain map $\cK: \mathds{R}^{d\times d}_{sym ,+} \rightarrow \mathds{R}^{d \times d}$ defined by
\begin{equation} 
\label{KalmanGain} 
\cK(P) := P H^T\left( R + H P H^T \right)^{-1},  
\end{equation}
and the $d\times M$-matrix 
\begin{equation}
E_k^a := \tau\left(E_k^f\right)
\end{equation}
is obtained via some nonlinear transformation  
\[\tau: \mathds{R}^{d\times M} \to \mathds{R}^{d\times M}\] 
of the centered forecast ensemble members such that the corresponding empirical covariance matrix 
$P_k^a$ satisfies the identity  
\begin{equation}
\label{eq:consistency} 
\begin{aligned} 
P_k^a & = \frac{1}{M-1} E_k^a \left( E_k^a\right)^T = \frac{1}{M-1} \tau\left(E_k^f\right)  
\left( \tau\left(E_k^f\right)\right)^T \\
&\overset{!}{=} \left({\rm Id} - \cK\left(P_k^{f}\right) H \right) P_k^f    
= P_k^f - P_k^f H^T\left( R + H P_k^f H^T \right)^{-1} H P_k^f   
\end{aligned} 
\end{equation} 
and $e_i$ denotes the $i$-th standard normal basis vector in $\mathds{R}^{M}$.

\subsubsection{Derivations of $\tau$}\label{DervTau}
The precise form of $\tau$ is specific to each ESRF algorithm. In this paper, we will consider three 
choices of $\tau$ surveyed in \cite{tippett2003}: 

\begin{itemize} 
\item[(i)] the Ensemble Adjustment Kalman Filter (EAKF) 
\begin{equation} 
\label{EAKFtau}
\tau\left(E_k^{f}\right) := A_kE_k^{f}
\end{equation}
for suitable transformations $A_k$ on the space of the signal process $X$,   
\item[(ii)] the Ensemble Transform Kalman Filter (ETKF) 
\begin{equation} 
\label{ETKFtau}
\tau\left(E_k^{f}\right) := E_k^{f} T_k
\end{equation}
for suitable transformations $T_k$ on the space of the ensemble members, 
\item[(iii)] and as a particular case of the adjustment filter, the unperturbed filter 
proposed in \cite{whitaker2002}, with transformation 
\begin{equation}
\label{whitaker1}
\tau\left(E_k^{f}\right) := \left({\rm Id} - \tilde{K}_kH\right)E_k^{f}
\end{equation}
where
\begin{equation}
\label{whitaker2} 
\tilde{K}_k := P_k^{f}H^T\left(R+HP_k^{f}H^T\right)^{-\frac{1}{2}}\left(\left( R+HP_k^{f}H^T\right)^{\frac{1}{2}} + R^{\frac{1}{2}}\right)^{-1}.
\end{equation}
\end{itemize} 

\medskip 
\noindent 
Two immediate choices for the respective transformation matrices $A_k$ and $T_k$ in the former two examples 
can be derived from the two canonical factorizations 
$$ 
P_k^f = \sqrt{P_k^f} \sqrt{P_k^f} = \frac 1{M-1} E_k^f (E_k^f)^T 
$$ 
of the empirical covariance matrix on the state space (resp. on the ensemble space), as follows: 

\begin{itemize} 
\item[(a)] Using the first factorization and using $\bar{\cK}_k := \cK\left(P_k^{f}\right)$ and 
$U := \sqrt{P_k^f} H^T R^{-\frac 12}$, the right hand side of the consistency condition 
\eqref{eq:consistency} can be written as 
$$ 
\begin{aligned} 
&({\rm Id}- \bar{\cK}_k H) P_k^f \\
& = P_k^f - P_k^f H^T \left(R + HP_k^f H^T\right)^{-1} HP_k^f\\
& = \sqrt{P_k^f} \left( {\rm Id} - U\left( {\rm Id} + U^T U  \right)^{-1} U^T \right)\sqrt{P_k^f} \\ 
& = \sqrt{P_k^f} \left({\rm Id} + UU^T \right)^{-1} \sqrt{P_k^f} \\ 
& = \sqrt{P_k^f} \left({\rm Id} + \sqrt{P_k^f} H^T R^{-1} H\sqrt{P_k^f}\right)^{-1}\sqrt{P_k^f}.
\end{aligned}
$$
This way we obtain
\begin{equation} 
\label{EAKFTrans}
A_k := \sqrt{P_k^f} \left({\rm Id} + \sqrt{P_k^f} H^T R^{-1}H \sqrt{P_k^f}\right)^{-\frac{1}{2}} 
\sqrt{P_k^f}^{-1}
\end{equation}
corresponding to equation (18) in \cite{tippett2003}. $\sqrt{P_k^f}^{-1}$ denotes the pseudo inverse of 
$\sqrt{P_k^f}$. 


\item[(b)] Using the second factorization $P_k^f = E_k^f (E_k^f)^T$ on the ensemble space, where we let 
$E_k^f$ denote the centered ensemble members weighted by $\frac 1{\sqrt{M-1}}$ for notational ease,   
and this time $V := (E_k^f)^T H^T R^{-\frac 12}$, the right hand side of the consistency condition 
\eqref{eq:consistency} can be written as 
$$ 
\begin{aligned}
& \left({\rm Id}- \bar{\cK}_k H\right) P_k^f \\
& = E_k^f \left( {\rm Id}-  (E_k^f)^T H^T \left( R + HE_k^f (E_k^f)^T H^T\right)^{-1} 
  HE_k^f \right) (E_k^f)^T \\
& = E_k^f \left( {\rm Id}- V \left( {\rm Id} + V^T V\right)^{-1} V^T \right) \\
& = E_k^f \left(I + VV^T \right)^{-1} \left(E_k^{f}\right)^T 
\end{aligned} 
$$ 
which gives
\begin{equation} 
\label{ETKFTrans}
T_k :=\left({\rm Id} + VV^T\right)^{-\frac{1}{2}}
= \left( {\rm Id} + \left(E_k^f \right)^T H^T R^{-1} H E_k^f \right)^{-\frac{1}{2}} 
\end{equation}
corresponding to equation (16) in \cite{tippett2003}. 
\end{itemize} 

\noindent
It is worth to note that the transformations $A_k$ and $T_k$ are adjoint in the sense that 
\begin{equation} 
\label{adjointness} 
A_k E_k^f = E_k^f T_k . 
\end{equation} 
This is a consequence of the following integral representation 
\begin{equation}
\label{eq:SquareRoot}
\sqrt{P^{-1}} = \frac{1}{\sqrt{\pi}}\int_0^{\infty} \frac{1}{\sqrt{t}}e^{-tP}{\rm d}t.
\end{equation} 
for any symmetric positive definite matrix $P$. Indeed, if $u$ is an eigenvector of $P$ with eigenvalue 
$\lambda > 0$, then $e^{-tP}u = e^{-t\lambda}u$ and hence
\begin{equation*}
\frac{1}{\sqrt{\pi}}\int_0^{\infty} \frac{1}{\sqrt{t}}e^{-tP}{\rm d}t u 
= \frac{1}{\sqrt{\pi}}\int_0^{\infty} \frac{1}{\sqrt{t}}e^{-t\lambda}{\rm d}t u 
= \frac{1}{\sqrt{\lambda}}u.
\end{equation*}
Using \eqref{eq:SquareRoot} and 
$$ 
e^{-t\left( {\rm Id} + \sqrt{P_k^f}H^TR^{-1}H \sqrt{P_k^f}\right)} =  e^{-t} e^{-t\sqrt{P_k^f}H^TR^{-1}H 
\sqrt{P_k^f} } 
$$ 
$$
\text{resp. }  
e^{-t\left( {\rm Id} + (E_k^f)^T H^T R^{-1}H E_k^f\right)} =  e^{-t} e^{-t(E_k^f)^T H^T R^{-1}H E_k^f} , 
$$ 
we can then write 
\begin{equation} 
\label{IntegralTransformA}
\begin{aligned}  
\left({\rm Id} + \sqrt{P_k^f} H^T R^{-1}H \sqrt{P_k^f}\right)^{-\frac{1}{2}} 
& = \frac{1}{\sqrt{\pi}} \int_0^{\infty} \frac{e^{-t}}{\sqrt{t}} 
e^{-t\sqrt{P_k^f}H^TR^{-1}H \sqrt{P_k^f}}{\rm d}t, \\ 
\left( {\rm Id} + \left(E_k^f \right)^T H^T R^{-1} 
H E_k^f \right) ^{-\frac{1}{2}} 
& = \frac{1}{\sqrt{\pi}} \int_0^{\infty} \frac{e^{-t}}{\sqrt{t}} 
e^{-t(E_k^f)^T H^T R^{-1}H E_k^f}{\rm d}t .
\end{aligned} 
\end{equation}
\eqref{IntegralTransformA} then implies that 
$$ 
\begin{aligned} 
A_k E_k^f 
& = \sqrt{P_k^f} \left({\rm Id} + \sqrt{P_k^f} H^T R^{-1}H \sqrt{P_k^f}\right)^{-\frac{1}{2}} 
\sqrt{P_k^f}^{-1} E_k^f \\ 
& =  \frac{1}{\sqrt{\pi}} \int_0^{\infty} \frac{e^{-t}}{\sqrt{t}} 
\sqrt{P_k^f}e^{-t\sqrt{P_k^f}H^TR^{-1}H \sqrt{P_k^f}} \sqrt{P_k^f}^{-1} {\rm d}t E_k^f \\ 
& = \frac{1}{\sqrt{\pi}} \int_0^{\infty} \frac{e^{-t}}{\sqrt{t}} 
e^{-tP_k^f H^TR^{-1}H} {\rm d}t\, E_k^f 
\end{aligned} 
$$ 
and on the other hand, 
$$ 
\begin{aligned} 
E_k^f T_k 
& = E_k^f \left({\rm Id} + \left(E_k^f \right)^T H^T R^{-1} H E_k^f \right) ^{-\frac{1}{2}} 
\\ 
& = \frac{1}{\sqrt{\pi}} \int_0^{\infty} \frac{e^{-t}}{\sqrt{t}} 
E_k^f e^{-t(E_k^f)^T H^T R^{-1}H E_k^f}  {\rm d}t \\ 
& = \frac{1}{\sqrt{\pi}} \int_0^{\infty} \frac{e^{-t}}{\sqrt{t}} 
e^{-tP_k^fH^TR^{-1}H} {\rm d}t\,  E_k^f  
\end{aligned} 
$$ 
which proves \eqref{adjointness}. 

\begin{remark}
\label{remarkNonUn}
It is obvious that the consistency condition \eqref{eq:consistency} does not uniquely determine 
$\tau$. In particular, \eqref{eq:consistency} is invariant w.r.t. orthonormal transformations of the 
forward ensembles, i.e. multiplication with orthonormal matrices from the right (see e.g. 
\cite{tippett2003}, \cite{wang2004}, \cite{livings2008}). Throughout the paper, we will only consider the 
original, unmodified algorithms and leave this additional degree of freedom for future research.
\end{remark}

\subsubsection{Choice of $\tau$} 
The above analysis of the adjustment filter and the transform filter motivates our following structural 
assumption on $\tau$ that will be the basis of our subsequent analysis. Assume that $\tau$ is of the 
following form
\begin{equation}\label{RepresentationTransform}
E_k^{a} = \tau\left(E_k^{f}\right) =: \cT\left(P_k^{f}\right)E_k^{f}
\end{equation}
where 
$$ 
\cT : \mathbb{R}^{d\times d}_{sym ,+} \to \mathbb{R}^{d\times d}  
$$
is defined on positive semidefinite symmetric matrices, satisfying the following
\begin{assum} 
\label{AssumT}
There exist positive constants $\cC^{\cT}$ and $\cL^{\cT}$ such that 
\begin{equation}
\label{assumT1}
\left\|\cT(P)\right\| \leq \cC^{\cT}\left( 1 + \|P\|\right) 
\end{equation}
\begin{equation}
\label{assumT2}
\left\|\cT(P) - \cT(Q)\right\| \leq \cL^{\cT}\left( 1 + \|P\|^2 + \|Q\|^2\right) \left\|P-Q\right\|
\end{equation} 
for all $P, Q \in \mathds{R}^{d\times d}_{sym ,+}$.  
\end{assum}

\begin{lemma} 
\label{VerifyAssum}
For the EAKF and the ETKF we have that  
\begin{equation*}
\cT(P) := \frac{1}{\sqrt{\pi}} \int_0^{\infty} \frac{e^{-t}}{\sqrt{t}}e^{-tPH^TR^{-1}H}{\rm d}t 
\end{equation*} 
which satisfies Assumption \ref{AssumT}. Furthermore, the unperturbed filter by \cite{whitaker2002}  
also satisfies Assumption \ref{AssumT}.
\end{lemma}
\noindent
For the proof see Appendix \ref{appVerifyAssum}.

\subsection{Continuous time}\label{algoCont}
\noindent
Now consider the continuous-time setting
\begin{align}
{\rm d}X_t &= B\left(X_t\right){\rm d}t + C{\rm d}W_t,\label{ContX}\\
{\rm d}Y_t &= HX_t{\rm d}t + \Gamma{\rm d}V_t\label{ContY}
\end{align}
where $W$, $V$ are independent Brownian motions. In this case, we analyze the corresponding filtering 
algorithm called the Ensemble Transform Kalman-Bucy Filter given by a continuous-time ensemble $X_t^{(i)}$, 
$1 \leq i \leq M$, $t \geq 0$, solving the following stochastic differential equation:
\begin{equation} 
\label{ContESRF}
{\rm d}X_t^{(i)} = B\left(X_t^{(i)}\right){\rm d}t + C{\rm d}W_t^{(i)} + P_tH^TR^{-1}\left({\rm d}Y_t
 - \frac{1}{2}H\left(X_t^{(i)} + \bar{x}_t\right){\rm d}t\right)
\end{equation}
with corresponding ensemble mean and covariance matrix
\begin{align}
\bar{x}_t &:= \frac{1}{M}\sum_{i=1}^MX_t^{(i)},\\
P_t &:= \frac{1}{M-1}\sum_{i=1}^M \left(X_t^{(i)} - \bar{x}_t\right)\left(X_t^{(i)} - \bar{x}_t\right)^T \label{PCont}
= \frac{1}{M-1}E_tE_t^T.
\end{align}
Existence of a global (strong) solution to \eqref{ContESRF} does not immediately follow from standard 
results in the theory of stochastic differential equations, since the coefficients are only locally 
Lipschitz and of cubic growth in the ensemble variables. However, it is possible to decompose  
\eqref{ContESRF} into a coupled system of stochastic differential equations for the centered particles 
$$ 
\hat{X}_t^{(i)} := X_t^{(i)} - \bar{x}_t , 1\le i\le M, 
$$  
and the empirical mean and to derive the following stochastic differential equation for $P_t$ up 
to the explosion time $\xi$ 
\begin{equation} 
\label{DynamicalEqEmpCovContESRF}
\begin{aligned} 
{\rm d} P_t  & = \frac 1{M-1} \sum_{i=1}^M \big( \left( B\left(X_t^{(i)}\right) - \bar{b}_t\right)  
\left( X_t^{(i)} - \bar{x}_t\right)^T  \\ 
& \hspace{4cm}  + \left( X_t^{(i)}- \bar{x}_t\right)\left( B\left(X_t^{(i)}\right)
- \bar{b}_t\right)^T \big) \, {\rm d}t \\ 
& \qquad + \left( Q - P_t H^T R^{-1} HP_t \right) \, {\rm d}t + {\rm d}N_t
\end{aligned} 
\end{equation} 
where $\bar{b}_t = \frac 1M \sum_{i=1}^M B\left(X_t^{(i)}\right)$ denotes the empirical mean of $B$ and 
$$ 
\begin{aligned} 
{\rm d} N_t & := \frac{1}{M-1}\sum_{i=1}^M \left(C{\rm d}\left(W_t^{(i)} - \bar{w}_t\right)\right)
\left(X_t^{(i)} - \bar{x}_t\right)^T \\ 
& \hspace{3cm}  + \left(X_t^{(i)} - \bar{x}_t\right)\left(C{\rm d}\left(W_t^{(i)} - 
\bar{w}_t\right)\right)^T  
\end{aligned} 
$$ 
is an It\^o stochastic integral, with $\bar{w}_t = \frac 1M \sum_{i=1}^M W_t^{(i)}$, hence a 
(matrix-valued) martingale w.r.t. the underlying filtration. We now obtain the following stochastic 
differential inequality for the trace 
$$ 
{\rm tr} (P_t) = \sum_{k=1}^d P_t (k,k) = \frac 1{M-1} \sum_{i=1}^M \left\|X_t^{(i)} - \bar{x}_t\right\|^2 
$$ 
of the empirical covariance matrix 
\begin{equation*} 
\begin{aligned} 
{\rm d}\, {\rm tr} (P_t)  
& \le \frac 2{M-1} \sum_{i=1}^M  \sum_{k=1}^d
 \left( B\left(X_t^{(i)}\right) - B(\bar{x}_t)\right) (k) \left( X_t^{(i)} - \bar{x}_t\right) (k) \, {\rm d}t \\ 
& \qquad + {\rm tr} (Q)\, {\rm d}t  + {\rm d}\, {\rm tr}(N_t) \\ 
& \le \left( 2\|B\|_{\rm Lip} {\rm tr} (P_t ) + {\rm tr} (Q)\right)\, {\rm d}t +  {\rm d}\, {\rm tr}(N_t),  
\end{aligned} 
\end{equation*} 
thereby using  
$$ 
\sum_{i=1}^M  \left( B\left(X_t^{(i)}\right) - \bar{b}_t\right) \left( X_t^{(i)} - \bar{x}_t\right)^T 
= \sum_{i=1}^M \left( B\left(X_t^{(i)}\right) - B(\bar{x}_t)\right)\left( X_t^{(i)} - \bar{x}_t\right)^T  . 
$$ 
The stochastic Gronwall lemma (see \cite{scheutzow2013}) now implies that 
\begin{equation} 
\label{TraceInequality} 
\E\left[ \sup_{0\le t\le T\wedge \xi} \sqrt{{\rm tr} (P_t)} \right] 
\le  \cC e^{ \|B\|_{{\rm Lip}} T} \sqrt{{\rm tr} (Q)}  . 
\end{equation} 
A similar analysis shows that 
$$
\E\left[ \sup_{0\le t\le T\wedge \xi} \sqrt{\|\bar{x}_t\|} \right] < \infty 
$$
which implies that $\P\left[ \xi \le T\right] = 0$ for all $T$. In conclusion, explosion does not occur 
in finite time a.s. which yields the a.s.-existence of a global solution of \eqref{ContESRF}.  

\medskip 
\noindent 
The ensemble $X_t^{(i)}$, $1 \leq i \leq M$, forms the so called continuous time limit of the ESRF 
algorithms presented in Section \ref{algoDiscr}, i.e. when applying the ESRF algorithms from Section 
\ref{algoDiscr} to the Euler approximation of \eqref{ContX}-\eqref{ContY} in discrete time and letting the 
discretization step converge to zero, the limiting ensemble evolves according to \eqref{ContESRF}. A 
fundamental observation in this respect is that the resulting ensemble algorithm in continuous time is the 
universal limit independent of the algorithm-specific transformation (see \cite{lange2019}). This follows 
from the fact that the transformations used in the ESRF, when applied to the Euler approximation, all agree 
up to first order w.r.t. the time-discretization.

\smallskip
\noindent
In the particular case of linear model drift $B(x) = Bx$, \eqref{ContESRF} yields the following system of evolution equations for $\bar{x}$ and $P$ in closed form:  
\begin{align}
{\rm d}\bar{x}_t &= B\bar{x}_t{\rm d}t + C{\rm d}\bar{w}_t + P_tH^TR^{-1}\left({\rm d}Y_t - H\bar{x}_t {\rm d}t\right),\\
{\rm d}P_t &= \left(BP_t + P_tB^T + Q - P_tH^TR^{-1}HP_t\right){\rm d}t + {\rm d}N_t.
\end{align}

\section{Mean field limit - discrete time}
\label{secMeanDiscr}

\noindent
We are now interested in investigating the asymptotic behaviour of the ESRF in the large ensemble limit 
$M \rightarrow \infty$. To this end assume that the initial conditions $X_0^{(i),a}$ are independent and 
identically distributed. It then follows from the representation \eqref{RepresentationTransform} that 
the joint distribution of the particles $X_k^{(i), f/a}$ is exchangeable. We can therefore expect to 
obtain a law of large numbers for mean and covariance and a propagation of chaos result. More specifically, 
the particles $X_k^{(i), f/a}$ converge to independent copies of a process $\bar{X}_k^{f/a}$ that is 
recursively defined as 
\begin{align} 
\bar{X}_k^{f} & := B\left(\bar{X}_{k-1}^{a}\right) + CW_k, \label{MeanFieldProcessF}\\
\bar{X}_k^{a} & := \bar{m}_k^a + \cT\left(\bar{P}_k^f\right)\left(\bar{X}_k^{f}  
- \bar{m}_k^f\right) \label{MeanFieldProcessA}
\end{align}
with  
\[
\bar{m}_k^a := \bar{m}_k^f + \cK\left(\bar{P}_k^{f}\right)\left( Y_k - H\bar{m}_k^f \right).\] 
Here,  
\begin{align*}
\bar{m}_k^f  & := \int x \bar{\pi}_k^f ({\rm d}x), \\
\bar{P}_k^f & := \int \left(x - \bar{m}_k^f\right)\left(x - \bar{m}_k^f\right)^T  
\bar{\pi}_k^f ({\rm d}x)
\end{align*}
denote the mean and covariance matrix of the distribution $\bar{\pi}_k^f$ of $\bar{X}_k^{f}$. The process 
$\bar{X}_k^{f/a}$ is called a mean-field process, since in the analysis step the transformation of the 
particle depends on its distribution.

\subsection{The linear model case}
\noindent
In the linear model case $B(x) = Bx$, the recursive equations \eqref{MeanFieldProcessF}-\eqref{MeanFieldProcessA} for the mean-field 
process reduce to 
\begin{align*}
\bar{X}_k^{f} & = B\bar{X}_{k-1}^{a} + CW_k, \\
\bar{X}_k^{a} & = \bar{m}_k^a + \cT\left(\bar{P}_k^f\right) \left(\bar{X}_k^{f} - \bar{m}_k^f\right), 
\end{align*} 
so that the recursive equations for the mean 
\begin{align*} 
\bar{m}_k^f &:= \mathds{E}\left[ \bar{X}_k^{f}\right] = B \mathds{E}\left[ \bar{X}_{k-1}^{a}\right] = B\bar{m}_{k-1}^a, \\
\bar{m}_k^a &:= \mathds{E}\left[ \bar{X}_k^{a}\right] = \bar{m}_k^f + \cK\left(\bar{P}_k^{f}\right)\left( Y_k - H\bar{m}_k^f \right) 
\end{align*} 
with covariance matrices
\begin{align*} 
\bar{P}_k^f & = B\bar{P}_{k-1}^a B^T + CC^T, \\ 
\bar{P}_k^a & = \cT\left(\bar{P}_k^f\right) \bar{P}_k^f\cT\left(\bar{P}_k^f\right)^T = \left({\rm Id} - \cK\left(\bar{P}_k^{f}\right) H\right)\bar{P}_k^f 
\end{align*}
coincide exactly with the recursive equations of mean and covariance of the Kalman filter. This implies in particular that for Gaussian initial distribution $\pi_0$ the distribution $\bar{\pi}_k^a$ of the mean-field process $\bar{X}_k^{a}$ coincides with the posterior distribution $\pi_k$ of $X_k$ given $Y_{1:k}$.

\subsection{Main results} 

To simplify notations in our subsequent analysis, let $\cT_k := \cT\left(P_k^{f}\right)$ 
and $\bar{\cT}_k := \cT\left(\bar{P}_k^{f}\right)$ as well as  
$\cK_k := \cK\left(P_k^{f}\right)$ and $\bar{\cK}_k := \cK\left(\bar{P}_k^{f}\right)$. Let $\bar{X}_k^{(i), f/a}$ be independent copies of the mean-field process 
$\bar{X}_k^{f/a}$ with $W_k$ replaced by $W_k^{(i)}$ in the forward step. We will show that 
\begin{equation*} 
X_k^{(i),f/a} = \bar{X}_k^{(i),f/a} + r_k^{(i),f/a}
\end{equation*}
where the residual $r_k^{(i),f/a}$ decays at rate $\frac{1}{\sqrt{M}}$ a.s. and also in $L^p$. This implies 
conversely that in the case where the mean-field process is consistent in the sense that $\bar{m}^a_k$ 
coincides with the respective posterior mean, the empirical mean $\bar{x}_k^a$ of the finite size ESRF 
provides a good approximation of the posterior mean. Moreover, our analysis of the residual $r_k^{(i), a}$ 
will provide us with explicit error estimates.\\
\noindent
We will follow closely the approach taken in \cite{leGland2009} to the mean field limit for general 
particle filters. Denote with 
\begin{equation*}
\Delta_k^{M,p,f/a} := \left(\frac{1}{M}\sum_{i=1}^M \left\|r_k^{(i),f/a}\right\|^p\right)^{\frac{1}{p}}
\end{equation*}
then our first main result is a follows:

\begin{theorem}
\label{thm:MeanFieldNonlinear}
Let $X_0^{(i),a}, i=1,...,M$, be an ensemble of i.i.d. random variables with distribution $\bar{\pi}_0$ 
having finite second moments and let $W_k^{(i)}$, $k\geq 1$, $i\geq 1$, be independent standard Gaussians. 
For $M \geq 2$, let $\left(X_k^{(i),a}\right)$, $i = 1,...,M,$ denote the ESRF filters defined by 
\eqref{AnalysisStep} - \eqref{KalmanGain} and \eqref{RepresentationTransform}, where $\cT$ 
satisfies Assumption \ref{AssumT}, initialized at $X_0^{(i),a}$. Further let $\left(\bar{X}_k^{(i),a}\right)$, $i\geq 1$, be independent copies of the mean field process 
\eqref{MeanFieldProcessF}-\eqref{MeanFieldProcessA} with $W_k$ replaced by $W_k^{(i)}$ and with initial 
condition $\bar{X}_0^{(i),a} = X_0^{(i),a}$. Then 
\begin{equation} 
\label{thm:MFNonlinear:result1}
\Delta_k^{M,p,f/a} \longrightarrow 0, M \rightarrow \infty, \mathds{P}\text{-a.s.}
\end{equation}
and for all $i= 1, 2, 3, \ldots $ 
\begin{equation} 
\label{thm:MFNonlinear:result2}
\left\|X_k^{(i),f/a} - \bar{X}_k^{(i),f/a}\right\| \longrightarrow 0, M \rightarrow \infty,  
\mathds{P}\text{-a.s.}.
\end{equation}
\end{theorem}
\noindent
We will need the following well-known controls on the Kalman gain transformation. For completeness we 
included a proof in the Appendix \ref{appTK-MeanF}.

\begin{lemma}
\label{TK-MeanF}
Consider the transformation
\[ 
\cK: \mathds{R}^{d\times d}_{sym ,+} \to \mathds{R}^{d\times d} 
\]
given by
\begin{equation}
\cK(P):= PH^T\left(R + HPH^T\right)^{-1}.
\end{equation}
(recall from (\ref{KalmanGain})). Then there exist positive constants $\cC^{\cK}$ and $\cL^{\cK}$ such that 
\begin{itemize}
\item $\left\|{\rm Id} - \cK(P)H \right\| \leq \cC^{\cK}\left( 1 + \|P\|\right)$,
\item $\left\|\cK(P) - \cK(Q)\right\| \leq \cL^{\cK} \left( 1 + \|P\|\right)\left\|P-Q\right\|$,
\end{itemize}
for all $P$, $Q\in\mathds{R}^{d\times d}_{sym ,+}$.
\end{lemma}
\begin{remark}
In the case of nonlinear observations, the Kalman gain transformation no longer acts on the covariance matrix but directly on the ensemble. Hence we may not expect to obtain the same estimates as above simplifying our proceeding analysis.
\end{remark}
\noindent
Furthermore we will make use of
\begin{lemma}
\label{PSigma}
It holds
\begin{itemize}
\item $\sup_{ M \geq 2} \left\|P_k^{f}\right\| < \infty$ $\mathds{P}\text{-a.s.}$,
\item $\left\|P_k^{f} - \bar{P}_k^{f,M}\right\|  
\leq 2\sqrt{\frac{M}{M-1}}\left(\left(\text{tr}\left(P_k^{f}\right)\right)^{\frac{1}{2}} + \left(\text{tr}\left(\bar{P}_k^{f,M}\right)\right)^{\frac{1}{2}}\right)\Delta_k^{M,2,f}$. 
\end{itemize}
Moreover, if $\left\|P_0^{a}\right\|_p := \mathds{E}\left[\left\|P_0^{a}\right\|^p\right]^{\frac{1}{p}} < \infty$, then
\[\sup_{M \geq 2}\left\|P_k^{f}\right\|_p < \infty.\]
\end{lemma}
\noindent
For the proof see Appendix \ref{appPSigma}.\\

\begin{proof} (of Theorem \ref{thm:MeanFieldNonlinear}) 
We first consider $M$ fixed. In the forecast step, we obtain the following estimate
\begin{equation}\label{DeltaFRecurs}
\begin{aligned}
\Delta_k^{M,p,f} & = \left(\frac{1}{M}\sum_{i=1}^M \left\|X_k^{(i),f} - \bar{X}_k^{(i),f}  
\right\|^p\right)^{\frac{1}{p}}\\
& = \left(\frac{1}{M}\sum_{i=1}^M \left\|B\left(X_{k-1}^{(i),a}\right) - B\left(\bar{X}_{k-1}^{(i),a}
\right)\right\|^p\right)^{\frac{1}{p}} \leq \|B\|_{\text{Lip}}\Delta_{k-1}^{M,p,a}.
\end{aligned}
\end{equation}
For the analysis step let us first introduce the centered particles 
$$ 
\hat{X}_k^{(i),f/a} := X_k^{(i),f/a} - \bar{x}_k^{(i),f/a} , \qquad 
\hat{\bar{X}}_k^{(i),f/a} := \bar{X}_k^{(i),f/a} - \bar{m}_k^{f/a} . 
$$ 
Then 
\begin{equation}
\label{XAnaRecurs_a}
\begin{aligned}
& \left\|\hat{X}_k^{(i),a} - \hat{\bar{X}}_k^{(i),a}\right\| 
  = \left\|\cT_k \hat{X}_k^{(i),f} - \bar{\cT}_k\hat{\bar{X}}_k^{(i),f}\right\| \\ 
& \hspace{1cm}\le \left\|\cT_k \left(\hat{X}_k^{(i),f} - \hat{\bar{X}}_k^{(i),f}\right) \right\|  
 + \left\|\left( \cT_k - \bar{\cT}_k\right) \hat{\bar{X}}_k^{(i),f}\right\| \\ 
& \hspace{1cm}\le \cC^{\cT}\left( 1 + \left\|P_k^f\right\|\right) 
\left\|\hat{X}_k^{(i),f} - \hat{\bar{X}}_k^{(i),f}\right\|  \\ 
& \hspace{1cm}\qquad + \cL^{\cT}\left( 1 + \left\|P_k^{f}\right\|^2 + \left\|\bar{P}_k^{f}\right\|^2\right) 
\left\|P_k^f - \bar{P}_k^f\right\| \left\|\hat{\bar{X}}_k^{(i),f}\right\|.  
\end{aligned}
\end{equation}
Similarly, 
\begin{equation}
\label{XAnaRecurs_b}
\begin{aligned}
\left\|\bar{x}_k^{a} - \bar{m}_k^{a}\right\| 
& = \left\|\left({\rm Id} - \cK_k H\right) \bar{x}_k^{f} - 
\left({\rm Id} - \bar{\cK}_k H\right) \bar{m}_k^{f}  
+ \left(\cK_k - \bar{\cK}_k  \right) Y_k\right\| \\ 
& \le \left\|\left({\rm Id} - \cK_k H\right) \left( \bar{x}_k^{f} - \bar{m}_k^{f}\right)\right\|   
+ \left\| \left(\cK_k - \bar{\cK}_k  \right) \left( Y_k - H\bar{m}_k^f\right)\right\| \\ 
& \le \cC^{\cK}\left( 1 + \left\|P_k^f\right\|\right)  \left\|\bar{x}_k^{f} - \bar{m}_k^{f}\right\| \\
& \qquad + \cL^{\cK}\left( 1 + \left\|P_k^f\right\|\right) \left\|P_k^f - \bar{P}_k^f\right\|   
  \left\| Y_k - H\bar{m}_k^f\right\|, 
\end{aligned}
\end{equation}
thereby using \eqref{TK-MeanF}. 
Combining \eqref{XAnaRecurs_a} and \eqref{XAnaRecurs_b} we obtain that 
\begin{equation} 
\label{XAnaRecurs_c}
\begin{aligned}
& \left\|X_k^{(i),a} - \bar{X}_k^{(i),a}\right\| 
\le \left\|\hat{X}_k^{(i),a} - \hat{\bar{X}}_k^{(i),a}\right\| 
 + \left\|\bar{x}_k^{a} - \bar{m}_k^{a}\right\| \\ 
 & \hspace{0.5cm} \le \cC^{\cT}\left( 1 + \left\|P_k^f \right\|\right) 
\left\|X_k^{(i),f} - \bar{X}_k^{(i),f}\right\| \\ 
&  \hspace{0.5cm}\qquad + \left( \cC^{\cT} + \cC^{\cK}\right) \left( 1 + \left\|P_k^f \right\|\right)  
\left\|\bar{x}_k^{f} - \bar{m}_k^{f}\right\| \\ 
&  \hspace{0.5cm}\qquad + \Big( \cL^{\cT} \left( 1 + \left\|P_k^f \right\|^2 + \left\|\bar{P}_k^f\right\|^2\right) 
\left\|\hat{\bar{X}}_k^{(i),f}\right\|  \\ 
&  \hspace{4.5cm} + \cL^{\cK} \left( 1 + \left\|P_k^f\right\|\right) \left\| Y_k - H\bar{m}_k^f\right\| \Big)    
\left\|P_k^f - \bar{P}_k^f\right\|.
\end{aligned} 
\end{equation} 
Let
$\bar{m}_k^{f,M}$ (resp. $\bar{P}_k^{f,M}$) denote empirical mean (resp. empirical covariance matrix) of 
$\bar{X}_k^{(i),f}$, $1 \leq i \leq M$. Then using 
\begin{equation*} 
\begin{aligned} 
\left\|\bar{x}_k^{f} - \bar{m}_k^{f}\right\| 
&\leq \left\|\bar{x}_k^{f} - \bar{m}_k^{f,M}\right\| + \left\|\bar{m}_k^{f,M} - \bar{m}_k^{f}\right\| \\
& \leq \left(\frac{1}{M}\sum_{i=1}^M \left\|X_k^{(i),f} - \bar{X}_k^{(i),f}\right\|^p\right)^{\frac{1}{p}}  
  + \left\|\bar{m}_k^{f,M} - \bar{m}_k^{f}\right\| \\ 
\text{ and } \left\|P_k^f - \bar{P}_k^f\right\| 
& \leq \left\|P_k^f - \bar{P}_k^{f,M}\right\| + \left\|\bar{P}_k^{f,M} - \bar{P}_k^f\right\| 
\end{aligned} 
\end{equation*}
yields  
\begin{equation} 
\label{DeltaARecurs}
\begin{aligned}
\Delta_k^{M,p,a}  
& \leq \left( 2\cC^{\cT} + \cC^{\cK}\right) \left( 1 + \left\|P_k^f\right\|\right) \Delta_k^{M,p,f}\\
& \hspace{0.5cm} +\left(  \cC^{\cT} + \cC^{\cK}\right) \left( 1 +  \left\|P_k^f\right\| \right)\left\|\bar{m}_k^{f,M} - \bar{m}_k^{f}\right\|\\
& \hspace{0.5cm} + \Big( \cL^{\cT} \left( 1 + \left\|P_k^f \right\|^2 + \left\|\bar{P}_k^f\right\|^2\right) 
\left( \frac 1M \sum_{i=1}^M \left\|\hat{\bar{X}}_k^{(i),f}\right\|^p \right)^{\frac 1p} \\ 
& \hspace{3cm} + \cL^{\cK} \left( 1 + \left\|P_k^f\right\|\right) \left\| Y_k - H\bar{m}_k^f\right\| \Big)\left\|P_k^f - \bar{P}_k^f\right\|.
\end{aligned}  
\end{equation}
In the case $p=2$, an application of Lemma \ref{PSigma} to $\left\|P_k^f - \bar{P}_k^{f,M}\right\|$ 
results in an estimate of the following form:
\begin{equation}
\label{DeltaARecurs2}
\Delta_k^{M,2,a} \leq C_k^{(1)} \Delta_{k-1}^{M,2,a} + C_k^{(2)} \left\|\bar{m}_k^{f,M}  
- \bar{m}_k^{f}\right\| + C_k^{(3)}\left\|\bar{P}_k^{f,M} - \bar{P}_k^f\right\|. 
\end{equation}
Using independence of the mean field particles $\bar{X}_k^{(i), f}$ we obtain
\begin{equation*}
\lim_{M\to\infty} \frac{1}{M}\sum_{i=1}^M\left\|\bar{X}_k^{(i),f}- \bar{m}_k^f \right\|^2 
= \text{tr}\left(\bar{P}_k^{f}\right) \hspace{0.5cm} \mathds{P}\text{-a.s.}.  
\end{equation*}
This, together with Lemma \ref{PSigma}, implies that limes superior for every coefficient in 
\eqref{DeltaARecurs2} is a.s.-bounded (in $M$). We also have the strong law of large numbers for 
empirical mean and covariance matrix of the mean field particles.  

\noindent 
For $p=2$, \eqref{thm:MFNonlinear:result1} therefore follows by simple induction w.r.t. $k$, the case $k=0$ being trivial, 
since $X_0^{(i),a} = \bar{X}_0^{(i),a}$. In particular, for all $k$,  
\begin{equation}\label{momentConv}
\bar{x}_k^f \rightarrow \bar{m}_k^f , \hspace{0.5cm} \left\|P_k^f - \bar{P}_k^f\right\| \rightarrow 0  
\hspace{0.5cm} \mathds{P}\text{-a.s.}.
\end{equation}
\eqref{thm:MFNonlinear:result2} follows from \eqref{XAnaRecurs_c} with a similar recursion. \\
In the case of $p>2$, now use \eqref{momentConv} in \eqref{DeltaARecurs} and proceed similarly which concludes the proof.
\end{proof}

\medskip
\noindent
Allowing for higher-order moments of the initial distribution in the above setting, we may even extend the result from Theorem \ref{thm:MeanFieldNonlinear} to hold in $L^p$. In order to prove this, we will need the following result:

\begin{lemma}
\label{LpBar}
If the initial distribution has finite moment of order $p$ for some $p \geq 2$, then the random vectors $\bar{X}_k^{(i),f/a}$ have finite moments of the same order.
\end{lemma}

\noindent
For the proof see Appendix \ref{appLpBar}. Denote
\begin{equation*}
D_k^{M,p,f/a} := \mathds{E}\left[\left|\Delta_k^{M,p,f/a}\right|^p\right]^{\frac{1}{p}}
\end{equation*}
then $L^p$-convergence holds in the following sense:

\begin{theorem} 
\label{error}
Assume that the initial distribution admits any moment of order $p \geq 2$. Then it holds
\begin{equation}
\sup_{M \geq 2} \sqrt{M} D_k^{M,p,f/a} < \infty
\end{equation}
for all $p \geq 2$.
\end{theorem}

\begin{proof}
As in \cite{leGland2009}, we prove this statement for all $p$ via induction over $k$. It holds trivially 
in the case of $k=0$ since $\Delta_0^{M,p,f/a} = 0$.

\noindent 
Using \eqref{DeltaFRecurs} in the forecast step, we obtain the estimate
\begin{equation*}
D_k^{M,p,f} \leq \|B\|_{\text{Lip}}D_{k-1}^{M,p,a}
\end{equation*}
and thus the assertion for $D_k^{M,p,f}$ with all $p$. 

\noindent
In the analysis step, we proceed from estimate \eqref{DeltaARecurs}. Taking expectation we obtain that 
\begin{equation}
\label{thm3.5:eq1}  
\begin{aligned} 
D_k^{M,p,a} 
& \le \left( 2\cC^{\cT} + \cC^{\cK}\right) \mathds{E}\left[   \left( 1 + \left\|P_k^f\right\|\right)^p  
 \left( \Delta_k^{M,p,f}\right)^p \right]^{\frac 1p} \\ 
& \hspace{0.5cm} + \left( \cC^{\cT} + \cC^{\cK}\right) \mathds{E} \left[  \left( 1 +  \left\|P_k^f\right\| \right)^p  
 \left\|\bar{m}_k^{f,M} - \bar{m}_k^{f}\right\|^p \right]^{\frac 1p} \\   
& \hspace{0.5cm} +\mathds{E} \Big[ \big( \cL^{\cT}\left( 1 + \left\|P_k^f \right\|^2 + \|\bar{P}_k^f\|^2\right) 
\left( \frac 1M \sum_{i=1}^M \left\|\hat{\bar{X}}_k^{(i),f}\right\|^p \right)^{\frac 1p} \\ 
& \hspace{2cm} + \cL^{\cK} \left( 1+ \left\|P_k^f\right\|\right) \left\| Y_k - H\bar{m}_k^f\right\| \big)^p \left\|P_k^f - \bar{P}_k^f\right\|^p 
  \Big]^{\frac 1p} \\ 
& =: (I) + (II) + (III), \text{ say.} 
\end{aligned}
\end{equation}
We now estimate the three terms separately. Clearly, 
\begin{equation} 
\label{eq:estI} 
\begin{aligned} 
(I) \le \left( 2\cC^{\cT} + \cC^{\cK}\right)\mathds{E}\left[ \left( 1 + \left\|P_k^f\right\|\right)^{2p} 
  \right]^{\frac 1{2p}}  D_k^{M, 2p, f} . 
\end{aligned} 
\end{equation}  
For the estimate of the next term we apply the $L^p$-law of large numbers  
(see Theorem 5.2 in \cite{kwiatkowski2015}) to get for all $\bar{p}\in [2, \infty)$ that
\begin{equation*}
\mathds{E}\left[\left\|\bar{m}_k^{f,M}-\bar{m}_k^{f}\right\|^{\bar{p}} \right]^{\frac{1}{\bar{p}}}  
= \left\|\frac{1}{M}\sum_{i=1}^M \bar{X}_k^{(i),f} - \bar{m}_k^{f}\right\|_{\bar{p}} 
\leq \frac{C_{\bar{p}}}{\sqrt{M}}\left\|\bar{X}_k^{(1),f}\right\|_{\bar{p}} 
\end{equation*}
and thus  
\begin{equation} 
\label{eq:estII} 
\begin{aligned} 
(II) & \le  \left( \cC^{\cT} + \cC^{\cK}\right) \mathds{E} \left[  \left( 1 +  \left\|P_k^f\right\| \right)^{2p}   
\right]^{\frac 1{2p}} \mathds{E} \left[ \left\|\bar{m}_k^{f,M} - \bar{m}_k^{f}\right\|^{2p} 
\right]^{\frac 1{2p}}  \\ 
& \le  \left( \cC^{\cT} + \cC^{\cK}\right)\mathds{E} \left[  \left( 1 +  \left\|P_k^f\right\| \right)^{2p}   
\right]^{\frac 1{2p}} \frac{C_{2p}}{\sqrt{M}} . 
\end{aligned} 
\end{equation}  
Finally, for the third term we can estimate similarly
\begin{equation} 
\label{eq:estIII} 
\begin{aligned} 
(III) & \le \mathds{E} \Big[ \big( \cL^{\cT} \left( 1 + \left\|P_k^f\right\|^2 + \left\|\bar{P}_k^f\right\|^2\right) 
\left( \frac 1M \sum_{i=1}^M \left\|\hat{\bar{X}}_k^{(i),f}\right\|^{p} \right)^{\frac 1p} \\ 
& \hspace{1cm} + \cL^{\cK}\left( 1 + \left\|P_k^f\right\|\right) \left\| Y_k - H\bar{m}_k^f\right\| \big)^{2p} \Big]^{\frac 1{2p}}
\mathds{E} \Big[ \left\|P_k^f - \bar{P}_k^f\right\|^{2p} \Big]^{\frac 1{2p}}. 
\end{aligned} 
\end{equation}  
Using Lemma \ref{PSigma} and the $L^p$-law of large numbers again, we can further estimate 
\begin{equation} 
\label{eq:estIIIa} 
\begin{aligned} 
\mathds{E} & \Big[ \left\|P_k^f - \bar{P}_k^f\right\|^{2p} \Big]^{\frac 1{2p}} 
\le \mathds{E} \Big[ \left\|P_k^f - \bar{P}_k^{f,M}\right\|^{2p} \Big]^{\frac 1{2p}} 
 + \mathds{E} \Big[ \left\|\bar{P}_k^{f,M} - \bar{P}_k^f\right\|^{2p} \Big]^{\frac 1{2p}} \\ 
& \qquad \le  2\sqrt{\frac{M}{M-1}}\mathds{E} \Big[ \left(\left(\text{tr}\left(P_k^{f}\right)\right)^{\frac{1}{2}} 
 + \left(\text{tr}\left(\bar{P}_k^{f,M}\right)\right)^{\frac{1}{2}}\right)^{2p} \Delta_k^{M,2p,f}   
   \Big]^{\frac 1{2p}} \\
& \qquad\qquad\qquad  + \frac{C_{2p}}{\sqrt{M}} \left\|\hat{\bar{X}}_k^{(1),f} \left(\hat{\bar{X}}_k^{(1),f}\right)^T \right\|_{2p}  
\end{aligned} 
\end{equation} 
and  
\begin{equation} 
\label{eq:estIIIaa} 
\begin{aligned} 
\mathds{E} & \Big[ \left(\left(\text{tr}\left(P_k^{f}\right)\right)^{\frac{1}{2}} 
 + \left(\text{tr}\left(\bar{P}_k^{f,M}\right)\right)^{\frac{1}{2}}\right)^{2p} 
   \left( \Delta_k^{M,2,f} \right)^{2p}  \Big]^{\frac 1{2p}} \\ 
& \le \mathds{E} \Big[ \left(\left(\text{tr}\left(P_k^{f}\right)\right)^{\frac{1}{2}} 
 + \left(\text{tr}\left(\bar{P}_k^{f,M}\right)\right)^{\frac{1}{2}}\right)^{4p} \Big]^{\frac 1{4p}}   
 D_k^{M,4p,f}.   
\end{aligned} 
\end{equation} 
Summarizing, we obtain an upper bound of the form  
\begin{equation} 
\label{thm3.5:eq2}
\begin{aligned} 
D_k^{M,p,a} \le \cC_k^{(1)} D_k^{M,4p,f} 
+ \frac{\cC_k^{(2)}}{\sqrt{M}} \left( \left\|\bar{X}_k^{(1),f}\right\|_{2p} 
+ \left\|\hat{\bar{X}}_k^{(1),f} \left(\hat{\bar{X}}_k^{(1),f}\right)^T \right\|_{2p} \right)  
\end{aligned}
\end{equation}  
with finite constants $\cC_k^{(i)}$ independent of $M$. It follows that the assertion holds 
for $D_k^{M,p,a}$ with all $p$, which concludes the induction step and proves the theorem. 
\end{proof}

\section{Mean field limit - continuous time}
\label{secMeanCont}

\noindent
Recall from Section \ref{algoCont} the continuous-time ESRF ensemble
\begin{equation}
\label{ESRF-nonl}
{\rm d}X_t^{(i)} = B\left(X_t^{(i)}\right){\rm d}t + C{\rm d}W_t^{(i)}  
+ P_t^MH^TR^{-1}\left({\rm d}Y_t - \frac{1}{2}H\left(X_t^{(i)} + \bar{x}_t\right){\rm d}t\right)
\end{equation}
for $1 \leq i \leq M$, where $P_t^M$ is defined by \eqref{PCont}. Provided, the algorithm is initialized with independent and identically 
distributed random variables, we can again expect a law of large numbers for mean and covariance 
and obtain convergence towards independent copies of the solution to the following stochastic 
differential equation 
\begin{equation}
\label{MF-nonl}
{\rm d}\bar{X}_t = B\left(\bar{X}_t\right){\rm d}t + C{\rm d}W_t + \bar{P}_tH^TR^{-1}\left({\rm d}Y_t  
- \frac{1}{2}H\left(\bar{X}_t + \bar{m}_t\right){\rm d}t\right)
\end{equation}
where
\begin{align}
\bar{m}_t & := \int x \bar{\pi}_t({\rm d}x),\\
\bar{P}_t & := \int \left( x  - \bar{m}_t\right)\left(x - \bar{m}_t\right)^T \bar{\pi}_t({\rm d}x)
\end{align}
with $\bar{\pi}_t := \mathds{P}^{-1} \circ \bar{X}_t$. 
\begin{remark}
Similar to the discrete case, $\bar{X}_t$ is called a mean-field process and similar to the case of 
\eqref{ESRF-nonl}, the existence of a (strong) solution to \eqref{MF-nonl} is not immediate. In the linear case provided the initial condition is Gaussian, $\bar{X}$ will remain a Gaussian process with finite first and second moment. The extension to the nonlinear case, however, requires an a priori estimate on $\bar{P}_t$. This is an open 
problem in the literature, since the stopping time argument given in the case of \eqref{ESRF-nonl} is not applicable to the mean-field case. We leave this problem open and rather assume from now on the existence of strong solutions to the respective mean-field equations. 
\end{remark}
\noindent 
Then we can show the following:
\begin{theorem}
\label{resultMeanCont}
Let $X_0^{(i)}$, $i= 1, \ldots , M$, be i.i.d. with distribution $\bar{\pi}_0$ having finite second 
moments and let $W_t^{(i)}$, $i\geq 1$, be independent Brownian motions. Let $X_t^{(i)}$, $i = 1,..., M$,  
denote the strong solution to (\ref{ESRF-nonl}) and suppose that there exists strong solutions 
$\bar{X}_t^{(i)}$, $i \geq 1$, to (\ref{MF-nonl}) with $W_t$ replaced by $W_t^{(i)}$, both initialized  
at $X_0^{(i)}$. Then 
\begin{equation}
X_t^{(i)} = \bar{X}_t^{(i)} + r_t^{(i)}
\end{equation}
where the residual $r_t^{(i)}$ satisfies almost surely with respect to the distribution of $Y$
\begin{equation}
\sup_{t\in [0,T]} \frac{1}{M}\sum_{i=1}^M \left\|r_t^{(i)}\right\|^2 \rightarrow 0, M \rightarrow \infty
\end{equation}
in probability for all $T\ge 0$.
\end{theorem}

\begin{remark} 
\label{remark:thm4_1} 
It is also possible to obtain error estimates similar to the time discrete case, but only up to certain 
stopping times. More specifically, let 
\begin{align}
\label{empiricalMFProcess} 
\bar{m}_t^{M} &:= \frac{1}{M}\sum_{i=1}^M \bar{X}_t^{(i)},\\
\bar{P}_t^{M} &:= \frac{1}{M-1}\sum_{i=1}^M \left(\bar{X}_t^{(i)} - \bar{m}_t^{M}\right) 
\left(\bar{X}_t^{(i)} - \bar{m}_t^{M}\right)^T
\end{align}
be the empirical mean and covariance of the mean field ensemble $\bar{X}_t^{(i)}$, $1\le i\le M$. 
Then 
\begin{equation}
\label{ErrorEstimates} 
\sup_{M \geq 2} \sqrt{M}\mathds{E}\left[\sup_{t \in [0, T \wedge \theta_n^M \wedge \bar{\theta}_n^{M}]} 
\frac{1}{M}\sum_{i=1}^M \left\|r_t^{(i)}\right\|^2\right]^{\frac{1}{2}} < \infty
\end{equation}
w.r.t. stopping times
\begin{align}
\theta_n^M & := \inf\left\{t \ge 0: {\rm tr}\left(P_t^M\right) \ge n\right\}, \label{StoppingTime1}\\
\bar{\theta}_n^M & := \inf\left\{ t \ge 0 : {\rm tr}\left(\bar{P}_t^M\right) \ge n\right\}. \label{StoppingTime2} 
\end{align}
Moreover, 
\begin{align}
& \limsup_{n \rightarrow \infty} \sup_{M \ge 2} 
\mathds{P}\left[\theta_n^M \leq T\right] = 0,
\label{Stop1}\\
& \limsup_{n \rightarrow \infty} \sup_{M \ge 2} 
\mathds{P}\left[\bar{\theta}_n^M 
\leq T\right] = 0. 
\label{Stop2}
\end{align}
\eqref{ErrorEstimates} follows from \eqref{ErrEst8} below, \eqref{StoppingTime1} from 
\eqref{TraceInequality}, since the constant $\cC$ in this inequality is independent of $M$ and 
$\xi = \infty$, and \eqref{StoppingTime2} follows from the same analysis leading to 
\eqref{TraceInequality}, but with $X_t^{(i)}$ replaced by the ensemble members $\bar{X}_t^{(i)}$ 
of the mean field process.  
\end{remark}

\noindent
To simplify the notation, we will write $P_t$ instead of $P_t^M$ in the following analysis.

\begin{proof} (of Theorem \ref{resultMeanCont})
The residual $r_t^{(i)}$ evolves according to
\begin{align*}
r_t^{(i)} & = \int_0^{t} B\left(X_s^{(i)}\right) - B\left(\bar{X}_s^{(i)}\right) \, {\rm d}s \\
& \qquad  + \int_0^t \left(P_s - \bar{P}_s\right)H^T R^{-1} \left( {\rm d}Y_s  
- \frac 12 H\left( \bar{X}_s^{(i)} + \bar{m}_s\right)\, {\rm d}s \right) \\ 
& \qquad - \frac 12 \int_0^t P_s H^T R^{-1}H\left( X_s^{(i)} - \bar{X}_s^{(i)} 
 + \left(\bar{x}_s - \bar{m}_s\right) \right) {\rm d}s , 
\end{align*}
so that, using  
\begin{equation*}
{\rm d}Y_t = HX_t^{\text{ref}}{\rm d}t + \Gamma{\rm d}V_t
\end{equation*}
where $X^{\text{ref}}$ denotes the reference trajectory generating $Y$,  
\begin{equation} 
\label{ErrEst1}
\begin{aligned}
\frac 1M \sum_{i=1}^M & \left\|r_t^{(i)}\right\|^2 \le t \int_0^{t} \left( 4\|B\|_{{\rm Lip}}^2 
 + 2\left\|P_s H^T R^{-1}H \right\|^2 \right) \left( \frac 1M \sum_{i=1}^M \left\|r_s^{(i)}\right\|^2 
 \right) \, {\rm d}s \\
&  + 8\left\| \int_0^t \left( P_s - \bar{P}_s \right) H^T R^{-1}\Gamma\,  {\rm d}V_s \right\|^2 \\ 
& + 2t\int_0^t \left\| H^T R^{-1}H \right\|^2 \left( \frac 1M \sum_{i=1}^M \left\|\bar{X}_s^{(i)}\right\|^2 
 + \left\|X_s^{\rm ref}\right\|^2 \right) \left\| P_s - \bar{P}_s \right\|^2 \, {\rm d}s  \\ 
&  + 2t\int_0^t \left\|P_s H^T R^{-1}H \right\|^2 \left\|\bar{x}_s - \bar{m}_s\right\|^2\,  {\rm d}s . 
\end{aligned}
\end{equation} 
Introducing the process  
\begin{equation} 
\label{CoefficientL} 
\begin{aligned}
\mathcal{L}_t^{(1)} & := \|B\|_{{\rm Lip}}^2  +  \left\|H^T R^{-1}H\right\|^2 \Big( \left\|P_t\right\|^2 + \frac 1M \sum_{i=1}^M \left\|\bar{X}_t^{(i)}\right\|^2 + \left\|X_t^{\rm ref}\right\|^2 \Big)
\end{aligned} 
\end{equation} 
we may estimate the third and forth summand under the integral by
\[\mathcal{L}_s^{(1)}\left(\left\|\bar{x}_s - \bar{m}_s\right\|^2   
 + \left\|P_s - \bar{P}_s\right\|^2\right)\]
and obtain
\begin{equation} 
\label{ErrEst2} 
\begin{aligned}
\frac 1M \sum_{i=1}^M \left\|r_t^{(i)}\right\|^2 
& \lesssim t\int_0^t \mathcal{L}_s^{(1)} \left( \frac 1M \sum_{i=1}^M \left\|r_s^{(i)}\right\|^2  
 + \left\|\bar{x}_s - \bar{m}_s\right\|^2   
 + \left\|P_s - \bar{P}_s\right\|^2 \right) \, {\rm d}s  \\
& \qquad + \left\| \int_0^t \left( P_s - \bar{P}_s \right) H^T R^{-1}\Gamma\,  {\rm d}V_s  
\right\|^2 . 
\end{aligned} 
\end{equation} 
We now consider the stopping time $\zeta_n := \theta^M_n \wedge\bar{\theta}^M_n$, where $\theta^M_n$ and 
$\bar{\theta}^M_n$ are defined in \eqref{StoppingTime1} and \eqref{StoppingTime2}. Taking supremum over $s\le t\wedge\zeta_n$ and then expectation w.r.t. the Brownian motion $V$ in (4.16) we obtain
\begin{equation}  
\label{ErrEst3}
\begin{aligned}
\E \left[ \sup_{s\le t\wedge\zeta_n} 
\frac 1M \sum_{i=1}^M \|r_s^{(i)}\|^2 \right] 
& \lesssim \E \Big[ t\int_0^{t\wedge\zeta_n} \mathcal{L}_s^{(1)} \big( \sup_{u\le s\wedge \zeta_n} 
\frac 1M \sum_{i=1}^M \|r_u^{(i)}\|^2    \\ 
& \qquad\qquad + \left\|\bar{x}_s - \bar{m}_s\right\|^2  
+ \left\|P_s - \bar{P}_s\right\|^2 \big)\, {\rm d}s \Big]   \\ 
& \quad + \E \left[ \sup_{s\le t\wedge \zeta_n} 
\left\| \int_0^s \left( P_u - \bar{P}_u \right) H^T R^{-1}\Gamma\,  {\rm d}V_u \right\|^2  \right] . 
\end{aligned}
\end{equation} 
Applying Doob's maximal inequality to the It\^o stochastic integral yields 
\begin{equation} 
\label{DoobsInequality} 
\begin{aligned}
\mathds{E} & \Big[ \sup_{s \le t \wedge \zeta_n} 
 \left\|\int_0^s \left(P_u - \bar{P}_u\right)H^TR^{-1}\Gamma{\rm d}V_u\right\|^2\Big] \\ 
& \qquad\le 4 \mathds{E} \left[\int_0^{t \wedge \zeta_n}  
{\rm tr} \left( (P_s - \bar{P}_s )H^TR^{-1} H(P_s - \bar{P}_s)\right)\, {\rm d}s\right] \\ 
&  \qquad \le  4\text{tr}\left(H^TR^{-1}H\right) \mathds{E} \left[ \int_0^{t\wedge \zeta_n}  
\left\|P_s - \bar{P}_s\right\|^2  {\rm d}s   \right].
\end{aligned}
\end{equation} 
Inserting \eqref{DoobsInequality} into \eqref{ErrEst3} yields 
\begin{equation}  
\label{ErrEst4}
\begin{aligned}
\E \left[ \sup_{s\le t\wedge\zeta_n} 
\frac 1M \sum_{i=1}^M \left\|r_s^{(i)}\right\|^2 \right] 
& \lesssim \E \Big[ \int_0^{t\wedge\zeta_n} \mathcal{L}_s^{(2)} \big( \sup_{u\le s\wedge \zeta_n} 
\frac 1M \sum_{i=1}^M \left\|r_u^{(i)}\right\|^2    \\ 
& \qquad \qquad + \left\|\bar{x}_s - \bar{m}_s\right\|^2  
+ \left\|P_s  - \bar{P}_s\right\|^2 \big)\, {\rm d}s \Big]  
\end{aligned}
\end{equation} 
with $\mathcal{L}_s^{(2)} := t\mathcal{L}_s^{(1)} + \text{tr}\left(H^TR^{-1}H\right)$. 

\medskip 
\noindent 
Note that 
\begin{equation} 
\label{ErrEst5} 
\begin{aligned}
\left\|\bar{x}_t - \bar{m}_t\right\| 
& \leq \left\|\bar{x}_t - \bar{m}_t^{M}\right\| + \left\|\bar{m}_t^{M} - \bar{m}_t\right\|\\
&  \leq \frac{1}{M} \sum_{i=1}^M \left\|X_t^{(i)} - \bar{X}_t^{(i)}\right\|  
+ \left\|\bar{m}_t^{M} - \bar{m}_t\right\| \\ 
& = \frac{1}{M} \sum_{i=1}^M \left\|r_t^{(i)}\right\| + \left\|\bar{m}_t^{M} - \bar{m}_t\right\|
\end{aligned} 
\end{equation} 
and, using as before the superscript $\hat{\ }$ to denote centered ensemble members, 
\begin{equation} 
\label{ErrEst6} 
\begin{aligned}
\left\|P_t - \bar{P}_t\right\| 
& \le \left\|P_t - \bar{P}^M_t\right\| + \left\|\bar{P}^M_t - \bar{P}_t\right\| \\ 
& \le \frac 1{M-1} \left\|\sum_{i=1}^M  \left( \hat{X}_t^{(i)}- \hat{\bar{X}}_t^{(i)} \right) 
\left(\hat{X}_t^{(i)}\right)^T  + \hat{\bar{X}}_t^{(i)} \left( \hat{X}_t^{(i)} 
- \hat{\bar{X}}^{(i)} \right)^T \right\| \\ 
& \qquad + \left\|\bar{P}^M_t - \bar{P}_t\right\| \\ 
& \le \frac 2{M-1} \sum_{i=1}^M \left\|r_t^{(i)}\right\| \left( \left\|\hat{X}_t^{(i)}\right\| + \left\|\hat{\bar{X}}_t^{(i)}\right\|\right) 
+ \left\|\bar{P}^M_t - \bar{P}_t\right\| \\ 
& \le 2\left( \frac 1{M-1} \sum_{i=1}^M \left\|r_t^{(i)}\right\|^2\right)^{\frac 12} 
\left( {\rm tr}(P_t) + {\rm tr}\left(\bar{P}_t^M\right)\right)^{\frac 12}  
+ \left\|\bar{P}^M_t - \bar{P}_t\right\| .   
\end{aligned} 
\end{equation} 
Inserting \eqref{ErrEst5} and \eqref{ErrEst6} into \eqref{ErrEst4} we obtain that
\begin{equation}  
\label{ErrEst7}
\begin{aligned}
\E \left[ \sup_{s\le t\wedge\zeta_n} 
\frac 1M \sum_{i=1}^M \left\|r_s^{(i)}\right\|^2 \right] 
& \lesssim \E \Big[ \int_0^{t\wedge\zeta_n} \mathcal{L}_s^{(3)} \Big( \sup_{u\le s\wedge \zeta_n} 
\frac 1M \sum_{i=1}^M \left\|r_u^{(i)}\right\|^2    \\ 
& \qquad \qquad + \left\|\bar{m}^M_s - \bar{m}_s\right\|^2  
+ \left\|\bar{P}^M_s  - \bar{P}_s\right\|^2 \Big)\, {\rm d}s \Big]  
\end{aligned}
\end{equation} 
with $\mathcal{L}_s^{(3)} := 4 \mathcal{L}_s^{(2)} \left( 1 +  {\rm tr} (P_s) + {\rm tr}\left(\bar{P}_s^M\right)\right)$. Now,   
$\mathcal{L}_s^{(3)}\le \cC n^2$ for some uniform constant independent of $n$ and $M$ and $s\le\zeta_n$, so that   
\eqref{ErrEst7} together with Gronwall's lemma yields that 
\begin{equation}  
\label{ErrEst8}
\begin{aligned}
\E & \left[ \sup_{s\le t\wedge\zeta_n} 
\frac 1M \sum_{i=1}^M \left\|r_s^{(i)}\right\|^2 \right] \\
& \qquad \le \cC n^2 \int_0^t e^{\cC n^2 s}\E \left[ \left\|\bar{m}^M_s - \bar{m}_s\right\|^2  
+ \left\|\bar{P}^M_s  - \bar{P}_s\right\|^2 \right] \, {\rm d}s \\ 
& \qquad = \cC n^2 \int_0^t e^{\cC n^2 s} \frac 1M \left( \left\|\bar{m}_s\right\|^2 + \left\|\bar{P}_s\right\|^2 \right) 
\, {\rm d}s \rightarrow 0 \, , M\to\infty  
\end{aligned}
\end{equation} 
due to the independence of the ensemble members of the mean field process. 
\end{proof}

\subsection{The linear model case}

\noindent
In the linear case $B(x) = Bx$, $\bar{m}_t$ and $\bar{P}_t$ evolve according to
\begin{align*}
{\rm d}\bar{m}_t &= B\bar{m}_t{\rm d}t + \bar{P}_tH^TR^{-1}\left({\rm d}Y_t - H\bar{m}_t{\rm d}t\right),\\
\frac{{\rm d}}{{\rm d}t} \bar{P}_t &= B\bar{P}_t + \bar{P}_tB^T + Q - \bar{P}_tH^TR^{-1}H\bar{P}_t
\end{align*}
which coincide with the Kalman-Bucy filtering equations for mean and covariance of the posterior 
distribution $\pi_t$ in the case of Gaussian initial conditions. Hence in this case Theorem 
\ref{resultMeanCont} in particular implies convergence of 
\begin{equation} 
\label{MFLimitMeanCovariance} 
\begin{aligned}
\sup_{0 \leq t \leq T} \left\|P_t - \bar{P}_t\right\| \longrightarrow 0, M \rightarrow \infty,\\
\sup_{0 \leq t \leq T} \left\|\bar{x}_t - \bar{m}_t\right\| \longrightarrow 0, M \rightarrow \infty
\end{aligned}
\end{equation} 
in probability, almost sure with respect to the distribution of $Y$, and thus in particular 
asymptotic consistency of the ESRF-algorithms. The asymptotic fluctuations in the convergence 
\eqref{MFLimitMeanCovariance} are extensively studied in the papers \cite{abishop2019}, 
\cite{abishop2019b}, and \cite{abishop2020} as summarized in the recent review paper \cite{abishop2020b}. 
There the authors provide $L^p$-estimates for the fluctuations of both empirical mean and covariance matrix 
of the continuous-time ESRF for arbitrary $M \geq 1$ around their respective mean field counterpart, and, 
as one would expect, again obtain a fluctuation rate of $\frac{1}{\sqrt{M}}$ (e.g. see Theorem 5.4 and 
Theorem 5.6 in \cite{abishop2020b}). 

\subsection{Associated stochastic partial differential equation}
\noindent 
Consider again the mean field limit
\begin{equation*}
{\rm d}\bar{X}_t = B\left(\bar{X}_t\right){\rm d}t + C{\rm d}W_t + \bar{P}_tH^TR^{-1}\left({\rm d}Y_t  
- \frac{1}{2}H\left(\bar{X}_t + \bar{m}_t\right){\rm d}t\right)
\end{equation*}
with $\bar{m}_t$ and $\bar{P}_t$ the mean vector and covariance matrix of $\bar{\pi}_t = \mathds{P} \circ  
\bar{X}_t^{-1}$ in the general nonlinear case. Let $\varphi \in C^2_0$ be a twice continuously 
differentiable function with compact 
support and let $\nabla\varphi$ (resp. $\varphi^{\prime\prime}$) be the gradient (resp. the Hessian)  
of $\varphi$. Applying It\^{o}'s formula yields 
\begin{align*}
{\rm d}\varphi & \left(\bar{X}_t\right) 
  = \nabla \varphi\left(\bar{X}_t\right){\rm d}\bar{X}_t 
 + \frac{1}{2} {\rm tr} \left( \varphi^{\prime\prime}\left(\bar{X}_t\right) {\rm d} 
  \left\langle \bar{X}, \bar{X}\right\rangle_t \right) \\
& = \nabla \varphi\left(\bar{X}_t\right)C{\rm d}W_t + L\varphi\left(\bar{X}_t\right){\rm d}t 
 + \frac{1}{2} {\rm tr} \left(\bar{P}_tH^TR^{-1}H\bar{P}_t \varphi^{\prime\prime} 
  \left(\bar{X}_t\right) \right) {\rm d}t \\
& \hspace{0.5cm} + \nabla \varphi\left(\bar{X}_t\right) \bar{P}_tH^TR^{-1}\left({\rm d}Y_t  
- \frac{1}{2}H\left(\bar{X}_t + \bar{m}_t\right){\rm d}t\right)
\end{align*}
where
\begin{equation*}
Lf(x) = \frac{1}{2}{\rm tr} \left(Qf^{\prime\prime}\right)(x) + B(x) \cdot \nabla f(x) 
\end{equation*}
and where we used
\begin{equation*}
{\rm d} \left\langle \bar{X}, \bar{X}\right\rangle_t 
= {\rm d}\left\langle \bar{P}_t H^T R^{-1} Y, \bar{P}_t H^T R^{-1}Y\right\rangle_t  
= \bar{P}_t H^T R^{-1} R R^{-1} H \bar{P}_t {\rm d}t.
\end{equation*}
Taking expectations yields the following mean-field stochastic partial differential equation  
\begin{equation}
\label{eq:mfSPDE} 
\begin{aligned} 
{\rm d}\int \varphi {\rm d}\bar{\pi}_t 
& = \int L\varphi {\rm d}\bar{\pi}_t {\rm d}t 
  + \frac{1}{2} \int {\rm tr} \left(\bar{P}_tH^TR^{-1}H\bar{P}_t \varphi^{\prime\prime} \right) 
    {\rm d}\bar{\pi}_t {\rm d}t \\
& \hspace{0.5cm} + \int \nabla \varphi {\rm d}\bar{\pi}_t \, \bar{P}_t H^TR^{-1} 
  \left({\rm d}Y_t - H\bar{m}_t{\rm d}t\right)\\
& \hspace{0.5cm} - \frac{1}{2} \int \nabla \varphi\,  \bar{P}_t H^TR^{-1}H\left(x - \bar{m}_t\right) 
{\rm d} \bar{\pi}_t {\rm d}t\\
&=: (I) + (II) + (III) + (IV)
\end{aligned}
\end{equation} 
which does not coincide with the Kushner-Stratonovich equation driving the posterior distribution. 
Having $\bar{\pi}_t = \mathcal{N}\left(\bar{m}_t, \bar{P}_t\right)$, hence $\nabla \bar{\pi}_t = 
- \bar{P}_t^{-1} (x- \bar{m}_t ) \bar{\pi}_t$, this however yields after partial integration
\begin{align*}
(II) 
& = \frac{1}{2}\sum_{k,l} \int \left(\bar{P}_t H^TR^{-1}H\bar{P}_t\right)_{k,l} 
\partial_{k,l} \varphi\,  {\rm d}\bar{\pi}_t\\
& = \frac{1}{2} \sum_{k,l} \int \left(\bar{P}_t H^TR^{-1}H\bar{P}_t\right)_{k,l} \partial_k  
   \varphi \left(\bar{P}_t^{-1}(x-\bar{m}_t)\right)_l {\rm d}\bar{\pi}_t\\
& = \frac{1}{2}\int \nabla\varphi \, \bar{P}_tH^TR^{-1}H  \left( x-\bar{m}_t\right) {\rm d}\bar{\pi}_t 
  = - (IV), 
\end{align*}
thus second and forth term cancel. Further
\begin{align*}
(III) 
& = \int \nabla\varphi {\rm d}\bar{\pi}_t\, \bar{P}_tH^TR^{-1} 
    \left({\rm d}Y_t - H\bar{m}_t{\rm d}t\right) \\
& = \int \varphi \langle \bar{P}_t^{-1}\left( x-\bar{m}_t \right) , \bar{P}_tH^TR^{-1}  
    \left({\rm d}Y_t - H\bar{m}_t{\rm d}t\right)\rangle {\rm d}\bar{\pi}_t \\
& = \int \varphi H(x-\bar{m}_t)^T {\rm d}\bar{\pi}_t R^{-1}\left({\rm d}Y_t - H\bar{m}_t{\rm d}t\right)
\end{align*}
yielding the innovation term in the Kushner-Stratonovich equation. 

\noindent 
Therefore, in the Gaussian case and after partial integration with respect to the distribution  
$\bar{\pi}_t$ of $\bar{X}_t$ we do obtain the Kushner-Stratonovich equation
\begin{equation}
\label{eq:KS} 
{\rm d}\int \varphi {\rm d}\bar{\pi}_t = \int L\varphi {\rm d}\bar{\pi}_t 
+ \int \varphi H(x-\bar{m}_t)^T {\rm d}\bar{\pi}_t R^{-1}\left({\rm d}Y_t - H\bar{m}_t{\rm d}t\right) 
\end{equation}
or in other words, in the linear case with Gaussian initial conditions, the solutions of the two 
stochastic partial differential equations \eqref{eq:mfSPDE} and \eqref{eq:KS} coincide.

\section{Conclusion and Outlook}
\noindent
In the setting of nonlinear, Gaussian systems given by linear measurements with Gaussian measurement error, we investigated convergence of the Ensemble Square Root filtering algorithms in the limit of increasing ensemble size both in discrete-time and continuous-time. We provide propagation of chaos results together with corresponding convergence rates and discuss consistency of the resulting mean field processes which is satisfied in the linear and Gaussian case.\\
Especially in the discrete-time setting, the presented analysis generates further open problems: as we already pointed out in Remark \ref{remarkNonUn}, the update equation for the empirical covariance matrix leaves room for various possible choices for the transformation of the centered ensemble members. In fact, the transformations we investigated in this paper are not unique in the sense that any orthogonal transformation $\mathcal{U}_k$ yields another valid transformation $\cT\left(P_k^{f}\right)E_k^{f}\mathcal{U}_k$  provided it is mean-preserving (see Remark \ref{remarkNonUn} for references).\\
Furthermore note that the our analysis covers adjustment filters and consequently carries over to their corresponding transform filters via the adjoint property discussed earlier. This identification, though analytically valid, may subsequently dismiss numerical properties of the transport filtering algorithms. The ETKF is mostly preferred over its ESRF counterparts due to its numerical performance: in fact, observe that the ETKF transformation $T_k$ is of order $M \times M$ thus in the regime of much smaller ensemble sizes than the state space dimension $d$, its computation is more feasible as opposed to the $d \times d$-transformations of the adjustment filters. It is not obvious how to identify a mean field process for the direct form of the ETKF in order to capture those benefits, which remains an open and interesting question for future research. 

\bigskip 
\noindent 
{\bf Acknowledgements} The research of Theresa Lange and Wilhelm Stannat has been partially funded by Deutsche Forschungsgemeinschaft (DFG) - SFB1294/1 - 318763901.

\appendix
\section{Proof of the Lemmata}

\noindent
\subsection{Proof of Lemma \ref{VerifyAssum}} 
\label{appVerifyAssum}
For the EAKF and the ETKF, the form of the transformation $\cT$ immediately follows 
from \eqref{IntegralTransformA}. It remains to verify Assumption \ref{AssumT}. To simplify 
notations let $\Theta := H^TR^{-1}H$. We will show that Assumption \eqref{assumT1} holds 
with 
\begin{equation} 
\label{Boundedness} 
\left\|\cT(P)\right\| \leq 1 + \frac{1}{2}\|\Theta\|\|P\|. 
\end{equation} 

\noindent 
Indeed, 
\begin{align*}
e^{-tP\Theta} - {\rm Id}  
= -\int_0^t e^{-sP\Theta}P\Theta{\rm d}s 
= -\int_0^t \sqrt{P}e^{-s\sqrt{P}\Theta\sqrt{P}}\sqrt{P}\, {\rm d}s\, \Theta .
\end{align*}
Since $\sqrt{P}\Theta\sqrt{P} \in \mathds{R}^{d\times d}_{sym ,+}$, we obtain that 
$\left\|e^{-s\sqrt{P}\Theta\sqrt{P}}\right\| \leq 1$, which implies
\begin{equation}
\label{BoundednessExp} 
\left\|e^{-tP\Theta} - {\rm Id}\right\| \leq \int_0^t \left\|\sqrt{P}\right\|^2{\rm d}s\, \|\Theta\|
= t\|P\|\|\Theta\|.
\end{equation}
Integrating up w.r.t. $t$ yields \eqref{Boundedness}, since  
\begin{equation*}
\begin{aligned} 
\left\|\cT(P)\right\| 
& \leq \frac{1}{\sqrt{\pi}} \int_0^\infty \frac{e^{-t}}{\sqrt{t}} 
\left\|e^{-tP\Theta} - {\rm Id}\right\|{\rm d}t + \|{\rm Id}\| \\ 
& \leq \frac{1}{\sqrt{\pi}} \int_0^{\infty} e^{-t} \sqrt{t}\, {\rm d}t\,  \|\Theta\|\|P\| + 1  
= \frac{1}{2}\|\Theta\|\|P\| + 1.
\end{aligned} 
\end{equation*}

\noindent
Assumption \eqref{assumT2} holds with
\begin{equation} 
\label{Lipschitz}
\left\|\cT(P)-\cT(Q)\right\|\ \leq \left( 1 + \|P\|^2 + \|Q\|^2 \right)\left(1 + \|\Theta\|^3 \right) 
\|P-Q\|. 
\end{equation} 

\noindent 
Indeed, first recall Duhamel's formula 
\begin{equation} 
\label{Duhamel} 
e^{-tP\Theta}-e^{-tQ\Theta} = \int_0^t e^{-(t-s) P\Theta}(Q-P)\Theta e^{-sQ\Theta}{\rm d}s
\end{equation} 
which implies that 
\begin{equation} 
\begin{aligned}
\left\| e^{-tP\Theta} \right.&\left. -e^{-tQ\Theta}\right\| 
\leq \int_0^t \|e^{-(t-s) P\Theta}\|\|Q-P\| \|\Theta \|\|e^{-sQ\Theta}\|\, {\rm d}s \\
& \le \int_0^t (1 + (t-s)\|P\Theta\| )( 1 + s \|Q\Theta\|) \, {\rm d}s \, 
\|\Theta \| \|P-Q\| \\ 
& = \left( t + \frac{t^2}{2} (\|P\Theta\| + \|Q\Theta\|) + \frac{t^3}{6}  \|P\Theta\|\|Q\Theta\|\right) 
\|\Theta \| \|P-Q\|. 
\end{aligned}
\end{equation} 
Integrating w.r.t. yields 
\begin{align*}
\big\|\cT(P) & -\cT(Q)\big\|
\le \frac 1{\sqrt{\pi}} \int_0^\infty \frac{e^{-t}}{\sqrt{t}} \left\| e^{-tP\Theta}-e^{-tQ\Theta}\right\| 
\, {\rm d}t \\
& \le \frac 1{\sqrt{\pi}} \int_0^\infty \frac{e^{-t}}{\sqrt{t}} 
\left( t + \frac{t^2}{2} (\|P\Theta\| + \|Q\Theta\|) + \frac{t^3}{6}  \|P\Theta\|\|Q\Theta\|\right) 
\, {\rm d}t \cdot \\ 
& \hspace{3cm} \cdot\|\Theta \| \|P-Q\| \\
& \lesssim \left( 1 + \|P\|^2 + \|Q\|^2 \right)\left(1 + \|\Theta\|^3 \right) \|P-Q\|, 
\end{align*}
proving \eqref{Lipschitz}. 

\medskip
\noindent
In case of the unperturbed filter in \cite{whitaker2002}, $\cT$ is given by
\begin{equation}
\label{Tstructure}
\cT(P) := {\rm Id} - PH^T \mathcal{R} \left( P\right) H.
\end{equation}
where 
\[ 
\mathcal{R}(P) := \left(R+HPH^T\right)^{-\frac{1}{2}}\left(\left(R+HPH^T\right)^{\frac{1}{2}} 
+ R^{\frac{1}{2}}\right)^{-1} \, , P\in \R^{d\times d}_{sym,+} . 
\]
Since $\left(R+HPH^T\right)^{\frac{1}{2}} \geq R^{\frac{1}{2}}$, we obtain that 
\begin{equation} 
\label{mathcalR}
\left\|\mathcal{R}(P)\right\| \leq \frac{1}{2}\left\|R^{-\frac{1}{2}}\right\|^2 = \frac 12 \left\|R^{-1}\right\|
\end{equation}
and thus Assumption \eqref{assumT1} is satisfied with
\begin{equation*}
\left\|\cT(P)\right\| \leq 1 + \|P\|\|H\|^2\left\|\mathcal{R}(P)\right\|  
\leq 1 + \frac{1}{2}\|P\|\|H\|^2\left\|R^{-1}\right\|.
\end{equation*}
Furthermore, using $A^{-1} - B^{-1} = -B^{-1}(A - B)A^{-1}$, one can show that
\begin{align*}
& \cT(P) - \cT(Q)\\
&= \left({\rm Id} - QH^T\mathcal{R}(Q)H\right)(P-Q)H^T\mathcal{R}(P) \\
&\hspace{0.5cm}- QH^T\mathcal{R}(Q)R^{\frac{1}{2}}\left( \left(R+HPH^T\right)^{\frac{1}{2}} - \left(R+HQH^T\right)^{\frac{1}{2}}\right)\mathcal{R}(P). 
\end{align*}
Using 
\begin{equation*}
\left(R + HPH^T\right)^{\frac{1}{2}} + \left(R + HQH^T\right)^{\frac{1}{2}} \geq 2R^{\frac{1}{2}}  
\end{equation*}
we may estimate
\begin{equation}
\begin{aligned}
& \left\|\left(R + HPH^T\right)^{\frac{1}{2}} - \left(R + HQH^T\right)^{\frac{1}{2}}\right\|\notag\\
& \hspace{1cm} \leq \frac{1}{2}\left\|R^{-\frac 12}\right\| \left\|H\left(P -Q\right)H^T\right\| 
\leq \frac{1}{2}\left\|R^{-\frac 12}\right\| \|H\|^2\left\|P - Q\right\|
\end{aligned} 
\end{equation}
(see \cite{vanHemmen1980}), which together with \eqref{mathcalR} yields Assumption \eqref{assumT2} with
\begin{align*}
&\left\|\cT(P) - \cT(Q)\right\|\\
&\leq\frac{1}{2}\|H\| \left\|R^{-1}\right\| \\
&\hspace{0.5cm}\times \left( \left\| {\rm Id} - QH^T\mathcal{R}(Q)H\right\| 
 + \left\|QH^T\mathcal{R}(Q)R^{\frac{1}{2}}\right\|  
 \frac{1}{2}\left\|R^{-\frac 12}\right\| \|H\| \right)\|P-Q\|.
\end{align*}
\qed
\subsection{Proof of Lemma \ref{TK-MeanF}} 
\label{appTK-MeanF}

The claim is easily checked by observing
\begin{equation*}
\begin{aligned} 
\left\|{\rm Id} - \cK(P)H\right\| 
& = \left\|{\rm Id} -  PH^T\left(R+HPH^T\right)^{-1}H\right\| \\
& \leq \left( 1 + \|P\| \|H\|^2  \left\|R^{-1}\right\|\right) 
\end{aligned} 
\end{equation*}
as well as
\begin{equation}\label{GainDiff}
\begin{aligned}
\cK(P) - \cK(Q) &= PH^T\left(R+ HPH^T\right)^{-1} - QH^T\left(R+HQH^T\right)^{-1}\notag\\
&= PH^T\left(R+HQH^T\right)^{-1}H\left(Q - P\right)H^T\left(R+HPH^T\right)^{-1} \notag\\
&\hspace{0.5cm} + \left(P - Q\right)H^T\left(R+HQH^T\right)^{-1}\\
&= \left({\rm Id} - \cK(P)H\right)\left(P - Q\right)H^T\left(R+ HQH^T\right)^{-1}
\end{aligned}
\end{equation}
where we used the standard result $A^{-1} - B^{-1} = -B^{-1}(A - B)A^{-1}$, 
yielding
\begin{equation*}
\begin{aligned} 
\left\|\cK(P) - \cK(Q)\right\| 
& \leq \left\|{\rm Id} - \cK(P)H\right\|\|H\|\left\|R^{-1}\right\|\|P -Q\| \\
& \leq \left(1 + \|P\| \|H\|^2 \left\|R^{-1}\right\|\right) \left\|R^{-1}\right\| \|P -Q\| .
\end{aligned} 
\end{equation*}
\qed

\subsection{Proof of Lemma \ref{PSigma}}\label{appPSigma}
For $v \in \mathds{R}^{d}$, $\|v\| =1$, we may estimate 
\begin{equation}
\begin{aligned} 
\left\langle P_k^{f} v, v\right \rangle 
& = \frac 1{M-1} \sum_{i=1}^M \left \langle X_{k}^{(i),f} - \bar{x}_{k}^{f}, v\right \rangle^2 \\ 
& = \frac 1{M-1} \sum_{i=1}^M \left \langle B\left( X_{k-1}^{(i),a}\right) - \bar{b}_{k-1}^a    
+ C \left(W_k^{(i)} - \bar{w}_k\right), v\right \rangle^2 \\ 
& \leq \frac{2}{M-1} \sum_{i=1}^M \left \langle B\left(X_{k-1}^{(i),a}\right) - \bar{b}_{k-1}^a, v  
\right \rangle^2 + \left \langle C \left( W_k^{(i)} - \bar{w}_k\right), v \right \rangle^2.
\end{aligned} 
\end{equation}
Here, $ \bar{b}_{k-1}^a := \frac 1M \sum_{i=1}^M B\left(X_{k-1}^{(i),a} \right)$ and 
$\bar{w}_k := \frac 1M \sum_{i=1}^M W_k^{(i)}$. Using
$$ 
\begin{aligned}
& \left \langle B\left(X_{k-1}^{(i),a}\right) - \bar{b}_{k-1}^{a}, v\right \rangle^2  
\leq \left\|B\left(X_{k-1}^{(i),a}\right) - \bar{b}_{k-1}^{a}\right\|^2\\
& \qquad\leq 2\left\|B\left(X_{k-1}^{(i),a}\right) - B\left(\bar{x}_{k-1}^{a}\right)\right\|^2  
+ 2\left\|B\left(\bar{x}_{k-1}^{a}\right) - \bar{b}_{k-1}^{a}\right\|^2\\
& \qquad\leq 2\left\|B\left(X_{k-1}^{(i),a}\right) - B\left(\bar{x}_{k-1}^{a}\right)\right\|^2  
+ \frac{2}{M}\sum_{j=1}^M \left\|B\left(X_{k-1}^{(j),a}\right) - B\left(\bar{x}_{k-1}^{a}\right) 
  \right\|^2\\
& \qquad\leq 2\|B\|_{\text{Lip}}^2\left(\left\|X_{k-1}^{(i),a} - \bar{x}_{k-1}^{a}\right\|^2  
+ \frac{1}{M}\sum_{j=1}^M \left\|X_{k-1}^{(j),a} - \bar{x}_{k-1}^{a}\right\|^2\right)
\end{aligned} 
$$
as well as
\begin{equation}
\label{trP}
\text{tr}\left(P_{k-1}^{a}\right)  
= \frac{1}{M-1}\sum_{i=1}^M \left\|X_{k-1}^{(i),a}  - \bar{x}_{k-1}^{a}\right\|^2,
\end{equation}
we obtain the estimate
\begin{equation}
\begin{aligned}
\left\langle P_k^{f} v, v\right\rangle 
&\leq 8\|B\|_{\text{Lip}}^2\text{tr}\left(P_{k-1}^{a}\right) 
+ \frac{2}{M-1}\left\langle C\mathcal{W}_k\mathcal{W}_k^T C^T v, v\right\rangle\\
&\leq 8\|B\|_{\text{Lip}}^2\text{tr}\left(P_{k-1}^{a}\right) 
+ \frac{2}{M-1}\left\| C\mathcal{W}_k\mathcal{W}_k^T C^T\right\| 
\end{aligned}  
\end{equation}
with 
$$ 
\mathcal{W}_k := \left[W_k^{(i)} - \bar{w}_k\right]_{i=1}^M. 
$$ 
Since $\text{tr} (P_{k-1}^a) \leq d \cdot \left\|P_{k-1}^a\right\|$, taking the supremum over all $v$ with $\|v\|=1$ yields the estimate
\begin{equation*}
\left\|P_k^{f}\right\| \leq 8d\|B\|_{\text{Lip}}^2\left\|P_{k-1}^{a}\right\| + \frac{2}{M-1}\left\| C\mathcal{W}_k\mathcal{W}_k^T C^T\right\|.
\end{equation*}
Together with
\begin{equation*}
P_k^{a} = \left({\rm Id} - \mathcal{K}\left(P_k^{f}\right)G\right)P_k^{f}  
= P_k^f - P_k^f H^T\left( R + H P_k^f H^T \right)^{-1} H P_k^f
\leq P_k^{f}
\end{equation*}
we may iterate over $k$ to obtain
\begin{equation}\label{NormPIter}
\left\|P_k^{f}\right\| \leq \left( 8d\|B\|_{\text{Lip}}^2 \right)^k \left\|P_0^a\right\|
 + \frac{2}{M-1} \sum_{l=1}^k \left( 8d\|B\|_{\text{Lip}}^2 \right)^{k-l}\left\|C\mathcal{W}_l\mathcal{W}_l^TC^T\right\|. 
\end{equation}
By independence of the $W^{(i)}$, the strong law of large numbers implies
\begin{equation*}
\begin{aligned} 
\frac{1}{M-1} C\mathcal{W}_k\mathcal{W}_k^T C^T  
& = C\left( \frac{1}{M-1} \sum_{i=1}^M \left(W_k^{(i)} - \bar{w}_k\right)  
 \left(W_k^{(i)} - \bar{w}_k\right) ^T \right) C^T \\ 
& \rightarrow C \mathds{E}\left[ \left(W_k^{(i)} - \bar{w}_k\right)  
 \left(W_k^{(i)} - \bar{w}_k\right)^T \right] C^T \\ 
& = d\, CC^T \quad \mathds{P}-\text{a.s.} 
\end{aligned} 
\end{equation*}
as $M$ tends to infinity, which gives
\begin{equation*}
\limsup_{M \rightarrow\infty} \left\|P_k^{f}\right\| 
\leq \left( 8d\|B\|_{\text{Lip}}^2 \right)^k \left\|P_0^a\right\| 
 + 2d \left\|CC^T\right\| \sum_{l=1}^k \left( 8d\|B\|_{\text{Lip}}^2 \right)^{k-l} 
\end{equation*}
and in particular
\begin{equation}
\sup_{M \geq 2}\left\|P_k^{f}\right\| < \infty \quad \mathds{P}-\text{a.s.}.
\end{equation}
Furthermore, from \eqref{NormPIter} we obtain
\begin{equation*}
\left\|P_k^{f}\right\|_p  
 \leq \left( 8d\|B\|_{\text{Lip}}^2 \right)^k \left\|P_{0}^a\right\|_p  
 + \frac{2}{M-1} \sum_{l=1}^k \left( 8d\|B\|_{\text{Lip}}^2 \right)^{k-l}  
 \left\| C\mathcal{W}_l\mathcal{W}_l^TC^T\right\|_p . 
\end{equation*}
Since 
\begin{align*}
&\frac{1}{M-1} \left\|C\mathcal{W}_k\mathcal{W}_k^TC^T\right\|_p \\
& = \left\|\frac{1}{M-1}\sum_{i=1}^M C\left(W_k^{(i)} - \bar{w}_k\right)\left(W_k^{(i)}  
-\bar{w}_k\right)^TC^T\right\|_p\\
& \leq \frac{M}{M-1} \left\|W_k^{(1)} - \bar{w}_k\right\|_{2p}^2  
\rightarrow \|C\|^2\left\|W_k^{(1)}\right\|_{2p}^2, M \rightarrow \infty,
\end{align*}
it therefore holds 
\[ 
\sup_{M \geq 2} \left\|P_k^{f}\right\|_p < \infty ,  
\]
provided $\left\|P_0^a\right\|_p < \infty$.

\medskip 
\noindent
For the proof of the second statement, using superscript $\hat{\ }$ to denote centered ensemble members, 
note that 
\begin{equation*} 
\begin{aligned}
&\left\|P_k^{f} - \bar{P}_k^{f,M}\right\|\\
& = \left\| \frac{1}{M-1} \sum_{i=1}^{M} \left(\hat{X}_k^{(i),f} - \hat{\bar{X}}_k^{(i),f} \right) 
\left(\hat{X}_k^{(i),f}\right)^T +  \hat{\bar{X}}_k^{(i),f} \left(\hat{X}_k^{(i),f}  
- \hat{\bar{X}}_k^{f}\right)^T\right\| \\
& \leq  \left( \left(\frac{1}{M-1} \sum_{i=1}^M  \left\|\hat{X}_k^{(i),f}\right\|^2\right)^{\frac{1}{2}} + 
\left( \frac{1}{M-1}\sum_{i=1}^M \left\|\hat{\bar{X}}_k^{(i),f} \right\|^2 \right)^{\frac{1}{2}}\right) \\ 
& \hspace{0.5cm} \times \left( \frac{1}{M-1}\sum_{i=1}^M \left\|\hat{X}_k^{(i),f}  
- \hat{\bar{X}}_k^{(i),f} \right\|^2 \right)^{\frac{1}{2}}\\
&\leq \left(\left(\text{tr}\left(P_k^{f}\right)\right)^{\frac{1}{2}} + \left(\text{tr}\left(\bar{P}_k^{f,M}
\right)\right)^{\frac{1}{2}}\right)\left(\frac{4}{M-1}\sum_{i=1}^M\left\|X_k^{(i),f} - \bar{X}_k^{(i),f}
\right\|^2\right)^{\frac{1}{2}}
\end{aligned}
\end{equation*}
where in the last line we used the Cauchy-Schwartz-inequality as well as identity \eqref{trP}.\\
\qed

\subsection{Proof of Lemma \ref{LpBar}}\label{appLpBar}
We proceed via induction over $k$. First, note that since $B$ is Lipschitz continuous there exists a 
constant $C_B$ such that
\[
\|B(x)\| \leq C_B(1+\|x\|). 
\]
Denote
\[
\bar{M}_k^{p,f/a} := \mathds{E}\left[\left\|\bar{X}_k^{(i),f/a}\right\|^p\right]^{\frac{1}{p}}, 
\]
then
\begin{equation*}
\bar{M}_k^{p,f} \leq C_B\left(1+\bar{M}_{k-1}^{p,a}\right) + \left\|CW_k^{(i)}\right\|_p 
\end{equation*}
which yields the claim for the forecast ensembles by the induction assumption. Further,
\begin{equation*}
\bar{X}_k^{(i),a} = \bar{m}_k^{f} + \bar{\cK}_k\left(Y_k - H\bar{m}_k^{f}\right) 
+ \bar{\cT}_k\left(\bar{X}_k^{(i),f} - \bar{m}_k^{f}\right)
\end{equation*}
implies
\begin{equation*}
\bar{M}_k^{p,a} \leq \left(1+\left\|\bar{\cT}_k\right\|\right)\left\|\bar{m}_k^{f}\right\| 
+ \left\|\bar{\cK}_k\right\|\left(\left\|Y_k\right\| + \|H\|\left\|\bar{m}_k^{f}\right\|\right) 
+ \left\|\bar{\cT}_k\right\|\bar{M}_k^{p,f}
\end{equation*}
yielding the claim for the analysis ensemble again by induction.\\
\qed

\begin{thebibliography}{99}

\index{References}

\bibitem{anderson2001}
Anderson, Jeffrey L.,
An ensemble adjustment Kalman filter for data assimilation.
Monthly weather review, Vol. 129, No. 12, 2884--2903, 2001

\bibitem{bergemann2012}
Bergemann, Kay, Reich, Sebastian,
An ensemble Kalman-Bucy filter for continuous data assimilation.
Meteorologische Zeitschrift, Vol. 21, No. 3, 213-219, 2012

\bibitem{beyou2013}
Beyou, Sebastien, Cuzol, Anne, Subrahmanyam Gorthi, Sai, Mémin, Etienne,
Weighted ensemble transform Kalman filter for image assimilation.
Tellus A: Dynamic Meteorology and Oceanography, Vol. 65, No. 1, 18803, 2013

\bibitem{abishop2019}
Bishop, Adrian N., Del Moral, Pierre,
On the stability of matrix-valued Riccati diffusions. 
Electronic Journal of Probability, Vol. 24, No. 24, 2019

\bibitem{abishop2019b}
Bishop, Adrian N., Del Moral, Pierre, Kamatani, Kengo, Remillard, Bruno,
On one-dimensional Riccati diffusions. 
Annals of Applied Probability, Vol. 29, No. 2, 1127--1187, 2019

\bibitem{abishop2020}
Bishop, Adrian N., Del Moral, Pierre, Niclas, Angele,
A perturbation analysis of stochastic matrix Riccati diffusions.
Annales de l’Institut Henri Poincaré, Probabilités et Statistiques, Vol. 56, No. 2, 884--916, 2020

\bibitem{abishop2020b}
Bishop, Adrian N., Del Moral, Pierre,
On the mathematical theory of ensemble (linear-Gaussian) Kalman-Bucy filtering.
arXiv:2006.08843v1, 2020

\bibitem{bishop2001}
Bishop, Craig H., Etherton, Brian J., Majumdar, Sharanya J.,
Adaptive sampling with the ensemble transform Kalman filter. Part I: Theoretical aspects.
Monthly weather review, Vol. 129, No. 3, 420--436, 2001

\bibitem{deWiljes2018}
de Wiljes, Jana, Reich, Sebastian, Stannat, Wilhelm,
{\em Long-time stability and accuracy of the ensemble Kalman--Bucy filter for fully observed processes and small measurement noise.}
SIAM Journal on Applied Dynamical Systems, Vol. 17, No. 2, 1152--1181, 2018

\bibitem{kalman1960}
Kalman, Rudolf E.,
A New Approach to Linear Filtering and Prediction Problems.
Journal of Basic Engineering, Vol. 82, No. 1, 35--45, 1960

\bibitem{kwiatkowski2015}
Kwiatkowski, Evan, Mandel, Jan,
Convergence of the Square Root Ensemble Kalman Filter in the Large Ensemble Limit.
SIAM/ASA Journal on Uncertainty Quantification, Vol. 3, No. 1, 1--17, 2015

\bibitem{lange2019}
Lange, Theresa, Stannat, Wilhelm,
On the continuous time limit of the Ensemble Square Root Filters.
arXiv:1910.12493v1, 2019

\bibitem{leGland2009}
Le Gland, Fran\c{c}ois, Monbet, Val\'erie, Tran, Vu-Duc,
Large sample asymptotics for the ensemble Kalman filter.
In: The Oxford handbook of nonlinear filtering, Eds. Crisan, et al., 
Oxford Univ. Press, Oxford, 598--631, 2011.

\bibitem{leeuwenburgh2005}
Leuuwenburgh, Olwijn, Evensen, Geir, Bertino, Laurent,
The impact of ensemble filter definition on the assimilation of temperature profiles in the tropical Pacific
Quarterly Journal of the Royal Meteorological Society, Vol. 131, 3291--3300, 2005

\bibitem{livings2008}
Livings, David M., Dance, Sarah L., Nichols, Nancy K.,
Unbiased Ensemble Square Root Filters.
Physica D: Nonlinear Phenomena, Vol. 237, 1021--1028, 2008

\bibitem{mandel2011}
Mandel, Jan, Cobb, Loren, Beezley, Jonathan D.,
On the convergence of the ensemble Kalman filter.
Applications of Mathematics, Vol. 56, No. 6, 533--541, 2011

\bibitem{okane2008}
O'Kane, Terence J., Frederiksen, Jorgen S.,
Comparison of statistical dynamical, square root and Ensemble Kalman Filters.
Entropy, Vol. 10, 684--721, 2008

\bibitem{scheutzow2013} 
Scheutzow, Michael, 
A stochastic Gronwall lemma. 
Infin. Dimens. Anal. Quantum Probab. Relat. Top., Vol. 16, 1350019, 4 pp, 2013.

\bibitem{tippett2003}
Tippett, Michael K., Anderson, Jeffrey L., Bishop, Craig H., Hamill, Thomas M., Whitaker, Jeffrey S.,
Ensemble Square Root Filters.
Monthly Weather Review, Vol. 131, 1485--1490, 2003

\bibitem{vanHemmen1980}
van Hemmen, J. Leo, Ando, Tsuneya,
An inequality for trace ideals. 
Communications in Mathematical Physics, Vol. 76, 143--148, 1980

\bibitem{wang2004}
Wang, Xuguang, Bishop, Craig H., Julier, Simon J.
Which is better, an ensemble of positive-negative pairs or centered spherical simplex ensemble?
Monthly Weather Review, Vol. 132, 1590--1605, 2004

\bibitem{whitaker2002}
Whitaker, Jeffrey S., Hamill, Thomas M.,
Ensemble data assimilation without perturbed observations
Monthly Weather Review, Vol. 130, No. 7, 1913--1924, 2002

\end{thebibliography}
\end{document}